\documentclass[12pt]{amsart}
\usepackage{amssymb, latexsym}

\theoremstyle{plain}
\newtheorem{theorem}{Theorem}[section]

\newtheorem{proposition}[theorem]{Proposition}
\newtheorem{lemma}[theorem]{Lemma}

\theoremstyle{definition}
\newtheorem{definition}[theorem]{Definition}
\newtheorem{remark}[theorem]{Remark}
\newtheorem{example}[theorem]{Example}

\numberwithin{equation}{section}

\usepackage{dsfont}
\usepackage{braket}
\usepackage{color}
\usepackage{tikz}
\usepackage{tkz-euclide}
\usepackage{tikz-3dplot}
\usepackage{graphicx}
\usepackage{hyperref}
\usepackage[color,matrix,arrow]{xy}
\input xy
\xyoption{all}

\begin{document}

\title[Convex polygon coordinates]{On the intrinsic geometry of polyhedra: Convex polygon coordinates}

\author[Romanowska]{A.B. Romanowska}
\author[Smith]{J.D.H. Smith}
\author[Zamojska-Dzienio]{A. Zamojska-Dzienio}
\address{(A.R., A.Z.) Faculty of Mathematics and Information Science\\
Warsaw University of Technology\\
00-662 Warsaw, Poland}

\address{(A.R., J.S.) Department of Mathematics\\
Iowa State University\\
Ames, Iowa, 50011, USA}

\email{(A.R.) anna.romanowska@pw.edu.pl\phantom{,}}
\email{(J.S.) jdhsmith@iastate.edu\phantom{,}}
\email{(A.Z.) anna.zamojska@pw.edu.pl}

\urladdr{(J.S.) 
\protect
{
\href{https://jdhsmith.math.iastate.edu/}
{https://jdhsmith.math.iastate.edu/}
}
}
\keywords{convex geometry; convex polytopes; convex polygons; barycentric coordinates; non-crossing chords; Catalan numbers}
\subjclass[2020]{52A99, 51M20, 52A10}

\begin{abstract}
There is a very extensive literature dealing with convex polytopes from the standpoints of combinatorics and numerical analysis. By contrast, the current paper adopts an alternative viewpoint that regards a polytope as an autonomous space in its own right, with its own intrinsic geometry. Our attention is focused on the complete set of all the coordinate systems that serve to locate a point of the polytope; divorced, for example, from the smoothness issues that are of concern for applications in numerical analysis. We use the efficient and appropriate algebraic language of barycentric algebras to elicit the convex structure of the set of polytope coordinate systems.

Specializing to convex polygons, we examine the chordal coordinate systems that are determined by the triangulations of the polygon. An algorithm to compute the coordinates of a point within such a system is presented. The algorithm relies on a coalgebra structure that transports a probability distribution on one side of a triangle to distributions on each of the remaining two sides. The Catalan number enumeration of the polygon triangulations (well-known within combinatorics) is then obtained in a natural geometric fashion from the parsing trees of the coalgebra structure of our algorithm.

\end{abstract}

\thanks{The work of the third author was partially conducted while a Visiting Scholar at Iowa State University. It was supported under a Fulbright Senior Award granted by the Polish-U.S. Fulbright Commission, and additionally by the Kosciuszko Foundation---The American Centre of Polish Culture.}

\maketitle

\tableofcontents

\section{Introduction}

Convex polyhedra appear as important objects in numerous areas of mathematics and its applications, and have been investigated using various techniques ranging from combinatorics to numerical analysis. The approach that is adopted in the current paper is to regard a convex polytope $\Pi$ as a geometric object in its own right, existing independently of the geometry of any ambient space. Our intrinsic approach is facilitated by the use of \emph{barycentric algebras}, which consist of binary operations that are defined for each element of the closed real unit interval $I$.

As a geometric space, a convex polytope $\Pi$ with vertex set $V$ has its own coordinate systems for locating its points as convex combinations of its vertices. The first main result of the paper, Theorem~\ref{T:PolCoSys}, shows that the full set $K_\Pi$ of all such coordinate systems forms a convex subset of the hypercube $I^{\Pi\times V}$. The proof of Theorem~\ref{T:PolCoSys} relies on a calculus for handling arbitrary coordinate systems, based on a version of Dirac's bra-ket notation.

The subsequent parts of the paper specialize to the case where $\Pi$ is a convex polygon. Using Theorem~\ref{T:PolCoSys} and our bra-ket calculus for coordinate systems, the second main result of the paper, Theorem~\ref{T:ChrdCord}, associates a system $\ket{\mathbf v}_\delta$ (for each $\mathbf v\in V$) of \emph{chordal coordinates} to each triangulation $\delta$ of $\Pi$ by a configuration of non-intersecting chords. These configurations appear naturally in the combinatorial setting, where it is known that they are counted by Catalan numbers. We present an algorithm to compute the chordal coordinates $\braket{\mathbf a|\mathbf v}_\delta$ of a point $\mathbf a$ of $\Pi$. The algorithm is based on a coalgebra operation which takes a probability distribution on one side of an oriented triangle, and transports it to a pair of probability distributions, one on each of the two remaining sides.  In our geometric approach, the Catalan enumeration follows naturally by assigning the triangulation $\delta$ to the parsing tree of the sequence of successive coalgebra operations created by the algorithm for $\ket{\mathbf v}_\delta$.

If the vertex set $V$ of the convex polygon $\Pi$ has size $n$, then the dihedral group $D_n$ of degree $n$ (and order $2n$) acts on $V$, inducing an action on the triangulations $\delta$ and their chordal coordinate systems $\ket{\mathbf v}_\delta$. The final main result of the paper, Theorem~\ref{T:CartoCDS}, demonstrates how averaging the chordal coordinate systems over the $D_n$-orbits provides more symmetrical coordinate systems, which we describe as \emph{cartographic}.

\subsection{Plan of the paper}

Barycentric algebras are defined in Section~\ref{S:BarycAlg} as sets equipped with a collection of basic binary operations indexed by the open real unit interval $I^\circ$, satisfying axioms known as \emph{idempotence} \eqref{E:idemptnc}, \emph{skew commutativity} \eqref{E:skewcomm}, and \emph{skew associativity} \eqref{E:skewasoc}. Algebraically, the ordered structures known as semilattices satisfy the axioms of idempotence, commutativity, and associativity. Thus, in a very strong sense, barycentric algebras should be viewed as geometrical expansions of semilattices. At the same time, the behavior of barycentric algebras shares much in common with the behavior of vector spaces in linear algebra.

Section~\ref{S:CoSyPoly} introduces the general framework for the study of barycentric coordinate systems on a given convex polytope $\Pi$, founded on the theory of barycentric algebras and the bra-ket notation \eqref{E:brac-ket} inspired by Dirac's notation for function application in quantum mechanics. Definition~\ref{D:PolCoSys} presents the formal specification of a coordinate system. The indicator or characteristic function of a subset $X$ of $\Pi$ is written as $\mathds 1_X$, while the identity function on $\Pi$ is written as $1_\Pi$. The bra-ket notation helps to clarify the distinction between the ``partition of unity'' property \eqref{E:BraketP1} as a partition of $\mathds 1_\Pi$ and the ``linear precision'' property \eqref{E:LinPrcsn} as a partition of $1_\Pi$.
Use of barycentric algebra shows that, in our context, the partition of unity property is actually a consequence of the linear precision property that does not need to be specified separately (Remark~\ref{R:PolCoSys}). Theorem~\ref{T:PolCoSys} shows that the set $K_\Pi$ of coordinate systems on a convex polytope $\Pi$ itself carries the structure of a convex set.
In the combinatorial setting, this convex set turns out to be the \emph{secondary polytope} (Remark~\ref{R:2ndaryPT}), see e.g. \cite[p.7, \S5.1]{DLRS} or \cite[Ch.~7]{GelKapZel}.

The conceptual framework developed in Section~\ref{S:CoSyPoly} plays a fundamental role in Sections~\ref{S:ChDeCaCo}--\ref{SS:CartoCor}, where our two basic coordinate systems for convex polygons are introduced: \emph{chordal coordinates} and \emph{cartographic coordinates}. The chordal coordinates are designed to be sparse, maximizing the number of weights that take the value zero. On the other hand, the cartographic coordinates are designed to be symmetrical, treating all parts of the convex polygon equally.

A given system of chordal coordinates for a convex $n$-gon $\Pi$ with vertex set $V$ is based on a decomposition of $\Pi$ into triangles by a set of non-crossing chords that connect pairs of vertices which are not adjacent on the boundary of $\Pi$. Figure~\ref{F:HexDcomp} illustrates the chordal decompositions of a hexagon. Each chordal decomposition is assigned a \emph{chordal degree sequence} (CDS): a partition of $n$ describing the degrees (or valencies) of the $n$-gon vertices in the graph formed by the chords. The dihedral group $D_n$ of degree $n$ and order $2n$ acts naturally on the set of vertices, thereby inducing permutations of the set of all chordal decompositions that share a given CDS.

The triangles appearing in a given chordal decomposition are known as its \emph{regions}, and the corresponding chordal coordinates of a point $\mathbf x$ of $\Pi$ are its uniquely defined barycentric coordinates with respect to the region(s) in which it lies. For a chordal decomposition, \S\ref{SSRegnIdOr} presents an algorithm to identify the regions, and to locate the region in which a given point $\mathbf x$ actually lies. The backbone of the algorithm is the rooted parsing tree for a non-coassociative coproduct (compare \cite[Def'n.~2.3(d)]{Sm171}) determined by each given chordal decomposition and initial edge (\S\ref{SSS:ChDePaTr}). An example from the hexagon is displayed in \S\ref{SSS:hexample}. Theorem~\ref{T:ChrdCord} verifies that the chordal coordinates furnish a coordinate system in the sense of Definition~\ref{D:PolCoSys}.

The cartographic coordinates are described in Section \ref{SS:CartoCor}. Briefly, for a given CDS, consider a chordal decomposition $\delta$ with that CDS. On the basis of Theorem~\ref{T:PolCoSys}, the cartographic coordinate system for the CDS is defined as the barycenter of the set of all chordal coordinate systems for chordal decompositions $\delta g$ in the orbit of $\delta$ under the action of group elements $g$ from the dihedral group $D_n$ (Definition~\ref{D:CartoCDS}).

\subsection{Notational conventions}

The essence of algebra lies in the repeated concatenation of successive functions, where the output of one function is fed as an input to the next. In this context, Euler's analytical notation as in the expression $\sin x$, where the argument is read after the function, clashes with the natural direction of reading from left to right.\footnote{Even within texts written in Semitic languages, mathematics is read in this direction.}
Hence, our default option for function notation is \emph{algebraic} or \emph{diagrammatic}, with functions following their arguments (e.g., $n!$ for the factorial function), or placed as a suffix (e.g., $x^2$ for the squaring function).

For other notational conventions and definitions that are not otherwise specified explicitly (e.g., for group actions), readers are invited to consult \cite{PMA}.

\section{Barycentric algebras}\label{S:BarycAlg}

\subsection{The basic definition}\label{SS:barycent}

Staring from the use of barycentric coordinates in geometry by M\"obius \cite{Moebius}, barycentric algebras were introduced in the nineteen-forties \cite{Kneser, Stone} for the axiomatization of real convex sets, presented algebraically with binary operations given by weighted means, the weights taken from the (open) unit interval in the real numbers.

Denote by $I$ the closed real unit interval $[0,1]$, and by $I^{\circ}$ the open real unit interval $I\smallsetminus\{0,1\}=]0,1[$.
For $p,q\in I^{\circ}$ define:
\begin{enumerate}
\item[$(\mathrm a)$] \emph{complementation}: $p'=1-p$;
\item[$(\mathrm b)$] \emph{dual product}: $p\circ q = p+q-p\cdot q=(p'q')'$.
\end{enumerate}
Note that the values $p'$ and $p\circ q$ are again in $I^{\circ}$.

\begin{definition}\cite{RS85,RS90,Modes}
A \emph{barycentric algebra} $A$ or $(A,I^\circ)$ is defined as a set $A$ that is equipped with binary \emph{operations}
\begin{equation}\label{E:OpOfBaAl}
\underline{p}\colon A\times A\to A;(x,y)\mapsto xy\,\underline{p}
\end{equation}
for each element (or \emph{operator}) $p\in I^\circ$. The operations \eqref{E:OpOfBaAl} are required to satisfy the properties of \emph{idempotence}
\begin{equation}\label{E:idemptnc}
xx\,\underline{p}=x
\end{equation}
for $x$ in $A$, \emph{skew-commutativity}
\begin{equation}\label{E:skewcomm}
xy\,\underline{p}=yx\,\underline{p'}
\end{equation}
for $x$, $y$ in $A$, and \emph{skew-associativity}
\begin{equation}\label{E:skewasoc}
xy\,\underline{p}\,z\,\underline{q}
=x\,yz\,\big(\,\underline{q/p\circ q}\,\big)\,\underline{p\circ q}
\end{equation}
for $x$, $y$, $z$ in $A$.
\end{definition}

\begin{example}\label{X:SemiLatt}
A \emph{semilattice} is a semigroup $(S,\cdot)$ which is commutative and idempotent. When each barycentric operation is interpreted as $xy\,\underline p=x\cdot y$ on a semilattice $S$, the semilattice becomes a barycentric algebra. In this case, skew-commutativity and skew-associativity reduce respectively to genuine commutativity and associativity.
\end{example}

\begin{example}\label{X:CoAfBAlg}
When $x$ and $y$ are elements of a real vector space $V$, and $p$ is an element of $I^\circ$, one can define
\begin{equation}\label{E:punderln}
xy\,\underline{p}=x(1-p)+yp=xp'+yp \, ,
\end{equation}
so that $\underline{p}$ is interpreted as a weighted mean operation combining the inputs $x$ and $y$. This way, one obtains the barycentric algebra $(V,I^\circ)$. Similarly, each convex set of the space $V$ forms a barycentric algebra.
\end{example}

\subsection{Subalgebras, homomorphisms, and products}\label{Ss:subalg}
\begin{definition}\label{D:SbWaSink}
Suppose that $(A,I^\circ)$ is a barycentric algebra, and $B$ is a subset of $A$.
If
$$
\forall\ p\in I^\circ\,,\
x\in B
\mbox{ and }
y\in B
\
\Rightarrow
\
xy\,\underline p\in B\,,
$$
then the subset $B$ is a \emph{subalgebra} of $(A,I^\circ)$.
\end{definition}

\begin{remark}\label{R:SbWaSink}
In Definition~\ref{D:SbWaSink}, the subalgebra $B$ is a barycentric algebra $(B,I^\circ)$ in its own right. Note also that the intersection of any family of subalgebras of a barycentric algebra $(A,I^\circ)$ is again a subalgebra of $(A,I^\circ)$.
\end{remark}

\begin{definition}
Let $(A,I^\circ)$ be a barycentric algebra, with a subset $S$. Then the \emph{subalgebra} $\braket S$ \emph{generated} by $S$ is the intersection of all the subalgebras of $(A,I^\circ)$ that contain $S$.
\end{definition}

\begin{example}
Let $S$ be a set of points in a real affine space $A$. Interpret $A$ as a barycentric algebra $(A,I^\circ)$ according to Example~\ref{X:CoAfBAlg}. Then the \emph{convex hull} of $S$ is the subalgebra $\braket S$ of $(A,I^\circ)$ that is generated by $S$.
\end{example}

\begin{definition}\label{D:BaHomSet}
Suppose that $(A,I^\circ)$ and $(A',I^\circ)$ are barycentric algebras.
\begin{enumerate}
\item[$(\mathrm a)$]
A function $f\colon A\to A';x\mapsto x^f$ is said to be a \emph{barycentric} (\emph{algebra}) \emph{homomorphism} if
\begin{equation}\label{E:BaryAHom}
xy\,\underline p^f=x^fy^f\,\underline p
\end{equation}
for all $x,y\in A$ and $p\in I^\circ$.
\item[$(\mathrm b)$]
A homomorphism is an \emph{isomorphism} if it is bijective.
\item[$(\mathrm c)$]
Write $\mathbf B(A,A')$ for the set of all barycentric homomorphisms from $(A,I^\circ)$ to $(A',I^\circ)$.
\end{enumerate}
\end{definition}

Informally stated, a barycentric homomorphism preserves the operations of the barycentric algebras. It preserves not only the basic binary operations $\underline{p}$, but also all the \emph{derived} operations that are obtained as compositions of the basic ones. For the barycentric algebras defined in Example \ref{X:CoAfBAlg}, these derived operations are just the \emph{convex combinations}.

\begin{remark}
Two composable barycentric homomorphisms compose to a barycentric homomorphism. This means that the variety $\mathbf B_0$ of barycentric algebras forms (the object class of) a category $\mathbf B$ of barycentric algebras, with the set $\mathbf B(A,A')$ as the set of morphisms from $A$ to $A'$.
\end{remark}

Now, we provide two crucial constructions for our further investigations. First, suppose that $X$ is a set and $(A',I^\circ)$ is a barycentric algebra. Consider the set $\mathbf{Set}(X,A')$ of all functions from $X$ to $A'$. For $p\in I^\circ$, define the \emph{pointwise} or \emph{componentwise} operation
\begin{equation}\label{E:pontwise}
fg\,\underline p\colon X\to A';x\mapsto x^fx^g\,\underline p
\end{equation}
on elements $f,g$ of $\mathbf{Set}(X,A')$. This way, we introduce the barycentric algebra structure on the set $\mathbf{Set}(X,A')$.

Further, take barycentric algebras $(A,I^\circ)$ and $(A',I^\circ)$. Again, $\mathbf B(A,A')$ becomes a barycentric algebra under the pointwise operations \eqref{E:pontwise}. It is a subalgebra of the barycentric algebra $\mathbf{Set}(A,A')$ \cite[\S 2.4.2]{RJSAZD} --- compare \cite[Prop.~5.1]{Modes}.

\begin{definition}
Suppose that $(A,I^\circ)$ and $(A',I^\circ)$ are barycentric algebras. Then their \emph{(direct) product} is the (barycentric) algebra defined on the set  $\set{(x,x')|x\in A\,,\ x'\in A'}$ with componentwise structure
$
(x,x')(y,y')\underline p=\left(xy\,\underline p,x'y'\,\underline p\right)
$
for all $p\in I^\circ$.
\end{definition}

Iterated products and powers are defined in the obvious way, since the construction of $(A,I^\circ)\times (A',I^\circ)$ extends naturally to a direct product of any number of factors.

The class $\mathbf{B}$ of barycentric algebras, being defined by identities, namely  \eqref{E:idemptnc}--\eqref{E:skewasoc}, forms a \emph{variety} of algebras --- a class closed under the taking of homomorphic images, subalgebras and direct products \cite[Th.~IV.2.3.3]{PMA}. This variety is generated by its \emph{cancellative} members, barycentric algebras which satisfy the additional property of \emph{cancellativity}:
$$
\forall\ p\in I^\circ\,,\
\forall\ x,y,z\in A\,,\
xy\,\underline p=xz\,\underline p
\
\Rightarrow
\
y=z.
$$
Cancellative barycentric algebras are \emph{convex sets} (\cite{N70}, \cite[Th.~269]{RS85}, \cite[Th.~5.8.6]{Modes}).

\subsection{Free algebras and simplices}\label{SS:barcoord}

Let $X$ be a set. Since the class $\mathbf B$ of barycentric algebras is a variety, there exists in $\mathbf B$ a free algebra $XB$ over the set $X$ \cite[p.308]{PMA}. This means that any mapping $f\colon X\rightarrow A$ from the set $X$ to the underlying set of any barycentric algebra $(A,I^\circ)$ has a unique extension to a barycentric homomorphism $\overline{f}\colon (XB,I^\circ)\rightarrow (A,I^\circ)$. Hence, every barycentric algebra is a homomorphic image of a free algebra. The free barycentric algebra over the empty set is empty.

\subsubsection{Construction of the free barycentric algebra}\label{SSS:FreeBAlg}

Given a set $X$, and the real line $(\mathbb R,I^\circ)$ as a barycentric algebra according to Example~\ref{X:CoAfBAlg}, take the barycentric algebra $\mathbf{Set}(X,\mathbb R)$. Consider the function
\begin{equation}\label{E:delta2xR}
X\to\mathbf{Set}(X,\mathbb R);
x\mapsto
\big[
\,
\delta_x\colon y\mapsto
\mbox{ \textbf{if} } y=x \mbox{ \textbf{then} } 1 \mbox{ \textbf{else} } 0
\,
\big]
\,.
\end{equation}
Define $XB$ as the subalgebra $\braket{\set{\delta_x|x\in X}}$ of $\mathbf{Set}(X,\mathbb R)$ generated by the set of delta functions.

By the Well-Ordering Theorem, the set $X$ may be given a total order $(X,\le)$. Repeated application of the axioms \eqref{E:idemptnc}--\eqref{E:skewasoc} for a barycentric algebra then shows that each element of $XB$ has a unique expression of the form
\begin{equation}\label{E:deltasqs}
\delta_{x_0}\delta_{x_1}\dots\delta_{x_r}\,\underline q_1\dots\underline q_r
\end{equation}
for some $r\in\mathbb N$, ordered subset $\set{x_0<x_1\dots<x_r}$ of $X$, and operators $q_1,\dots, q_r\in I^\circ$ (compare \cite[Lemma~5.8.1]{Modes}). The image of \eqref{E:deltasqs} under the extension $\overline f$ is then taken as the barycentric combination $x_0^fx_1^f\dots x_r^f\,\underline q_1\dots\underline q_r$ in $(A,I^\circ)$.

\subsubsection{Probability distributions and weighted averages}\label{SSS:PrDiWeAv}

The element \eqref{E:deltasqs} may be written as the convex combination
\begin{equation}\label{E:deltasps}
\sum_{k=0}^r\delta_{x_k}p_k
\end{equation}
with coefficients $p_k=q_kq_{k+1}'\dots q_r'\in I^\circ$ for $0\le k\le r$ and $q_0=1$. Thus, barycentric combinations may be taken as finitely supported probability distributions or weighted averages. In \eqref{E:deltasps}, weight $p_k$ is attached to the representation $\delta_{x_k}$ of $x_k$. The weights are barycentric coordinates in the sense of \cite[\S31]{Moebius}.

\subsubsection{Kernel functions}\label{SSS:KernlFun}
Let $A$ be a barycentric algebra generated by a finite set $V\subseteq A$. This means that each element $a$ of $A$ may be written as a barycentric combination
\begin{equation}\label{E:brycmbkn}
a=\sum_{v\in V}p(a,v)v
\end{equation}
with some specific choice of $p(a,v)\in I$, as observed in \S\ref{SSS:PrDiWeAv}.\footnote{The full real unit interval $I$ is required here, since some of the weights may be zero, and one of the weights may actually be $1$.}. Suppose that such choice has been done. One may then focus on the function
\begin{equation}\label{E:kernelfn}
p\colon A\times V\to I;(a,v)\mapsto p(a,v)
\end{equation}
which, in the language of analysis, may be called a \emph{kernel} (\emph{function}) due to its role in \eqref{E:brycmbkn}, or in continuous versions of \eqref{E:brycmbkn} such as \cite[\S3.1]{KosBarW}.

\subsubsection{Partitions of unity and linear precision}\label{SSS:P1LinPre}

The parametrized univariate (or \emph{curried} \cite[\S O.3.4]{PMA}) versions 
\begin{equation}\label{E:CoordFun}
A\to I;a\mapsto p(a,v)
\quad
\mbox{ for each }
\
v\in V
\end{equation}
of \eqref{E:kernelfn} are known as \emph{coordinate functions} (or ``barycentric coordinates'' \cite[\S1.1]{WarSchHirDes}). The equation
\begin{equation}\label{E:PartoOne}
\sum_{v\in V}p(a,v)=1
\end{equation}
is then expressed as saying that the coordinate functions form a \emph{partition of unity}, while \eqref{E:brycmbkn} is said to express the \emph{barycentric property} \cite[(3)]{KosBarW} or \emph{linear precision} \cite[\S1.1]{WarSchHirDes}. We record the following for subsequent reference.

\begin{lemma}\label{L:BPimpPO1}
If the barycentric property \eqref{E:brycmbkn} holds, then the partition of unity property \eqref{E:PartoOne} follows.
\end{lemma}

\begin{proof}
Suppose that \eqref{E:brycmbkn} holds for an element $a$ of $A$. Then, we may consider the constant barycentric homomorphism $k:A\to\set 1$. Thus \eqref{E:PartoOne} is obtained as the image of \eqref{E:brycmbkn} under $k$.
\end{proof}

\subsubsection{Interpolation}

A function $f\colon V\to C$ from $V$ to some codomain barycentric algebra $C$ may be \emph{extended} (in the algebraist's terminology) or \emph{interpolated} (in the analyst's terminology) to a function $\widehat f\colon A\to C$ with
$$
\widehat f(a)=\sum_{v\in V}p(a,v)f(v)
$$
making use of \eqref{E:brycmbkn} \cite[(2)]{WarSchHirDes}. The extension $\widehat f$ will be uniquely specified as the barycentric homomorphism $\overline f$ when  $A$ is freely generated by $V$.

\section{Coordinate systems for polytopes}\label{S:CoSyPoly}

In this chapter, we introduce our formalism for dealing with coordinate systems on polytopes. Beyond its general interest, the formalism will be needed for an efficient treatment of the chordal coordinates presented in Chapter~\ref{S:ChDeCaCo}.

\subsection{Function spaces}

Let $B$ be a barycentric algebra, e.g., a polytope. Let $\mathbb R^B$ denote the space $\mathbf{Set}(B,\mathbb R)$ of all real-valued functions on $B$. The space inherits pointwise algebra structure from any algebraic structures carried by $\mathbb R$, including commutative rings, real vector spaces, and barycentric algebras (in the latter case recalling the construction of the barycentric algebra $\mathbf{Set}(B,\mathbb R)$ in Subsection \ref{Ss:subalg}). We may also consider the barycentric subalgebra $I^B=\mathbf{Set}(B,I)$ consisting of those functions on $B$ that take values in the closed unit interval $I=[0,1]$. Typical elements of $I^B$ include the \emph{indicator functions}
\begin{equation}\label{E:IndctrFn}
\mathds 1_X\colon B\to I;
a\mapsto
\begin{cases}
1 &\mbox{if } a\in X\,;\\
0 &\mbox{otherwise}
\end{cases}
\end{equation}
of subsets $X$ of $B$.

\begin{lemma}\label{L:IdInRtoB}
$(\mathrm a)$
An element $e$ of the commutative ring $\mathbb R^B$ is \emph{idempotent} $(\mbox{i.e., }e^2=e)$ iff it is the indicator function $\mathds 1_X$ of a subset $X$ of $B$.
\vskip 2mm
\noindent
$(\mathrm b)$
For subsets $X,Y$ of $B$, we have
$
\mathds 1_X\cdot\mathds 1_Y=\mathds 1_{X\cap Y}
$
and
$
\mathds 1_X\circ\mathds 1_Y=\mathds 1_{X\cup Y}
$,
using the pointwise operations on $I^B$ inherited from $(I,\cdot,\circ)$.
\end{lemma}

\renewcommand{\arraystretch}{1.3}
\begin{table}[hbt]
\begin{tabular}{c||c|c|c|c}
&$\mathds 1_X$ &$\mathds 1_Y$ &$\mathds 1_{X\cap Y}$ &$\mathds 1_X\circ\mathds 1_Y$\\
\hline
\hline
$a\notin X\cup Y$ &$0$ &$0$ &$0$ &$0$\\
\hline
$a\in X\smallsetminus Y$ &$1$ &$0$ &$0$ &$1$\\
\hline
$a\in Y\smallsetminus X$ &$0$ &$1$ &$0$ &$1$\\
\hline
$a\in X\cap Y$ &$1$ &$1$ &$1$ &$1$\\
\hline
\hline
\end{tabular}
\vskip 2mm
\caption{Indicator function values used to verify the second statement of Lemma~\ref{L:IdInRtoB}(b).}
\label{T:aXYoneXY}
\end{table}

\begin{proof}
(a)
Note that an indicator function \eqref{E:IndctrFn} is an idempotent under pointwise multiplication. Conversely, suppose that $e$ is an idempotent of $\mathbb R^B$, with support $X=e^{-1}\left(\mathbb R\smallsetminus\set0\right)$. Since the value of $e$ at any point $a$ of $B$ is a root of the polynomial $x^2-x=x(x-1)$ over the field $\mathbb R$, it follows that $e=\mathds 1_X$.
\vskip 2mm
\noindent
(b)
First, note
$$
a\in X\cap Y
\
\Leftrightarrow
\
\left[\,a\mathds 1_X=1\mbox{ and }a\mathds 1_Y=1\,\right]
\
\Leftrightarrow
\
a(\mathds 1_X\cdot\mathds 1_Y)=1
$$
for $a\in B$. Recalling $\mathds 1_X\circ\mathds 1_Y=\mathds 1_X+\mathds 1_Y-\mathds 1_X\cdot\mathds 1_Y$, Table~\ref{T:aXYoneXY} may then be used to verify the second equation.
\end{proof}

\renewcommand{\arraystretch}{1}

\subsection{Barycentric coordinates in the function space}\label{SS:BaCoFuSp}

Suppose that $\Pi$ is a barycentric algebra generated by a finite set $V=\set{v_1,\dots,v_n}$. Then the basic barycentric coordinate concepts, as introduced in Section~\ref{SS:barcoord}, may be given natural formulations in the function space $\mathbb R^\Pi$. Here, the notation
\begin{equation}\label{E:brac-ket}
\Pi\times\mathbb R^\Pi\to\mathbb R\colon(a,f)\mapsto\braket{a|f}
\end{equation}
will be used to denote the evaluation of a function $f\in\mathbb R^\Pi$ at an element $a\in\Pi$. Inspired by P.A.M.~Dirac's ``bra-ket'' formalism, it will often be convenient to write $\ket f\in\mathbb R^\Pi$, and occasionally $\bra a\in\Pi$.

Assume, as in \S\ref{SS:barcoord}, that a system $p$ of barycentric coordinates is being presented, or determined. Thus each element $a$ of $\Pi$ may, or is to, be written as a barycentric combination
$a=\sum_{v\in V}p(a,v)v$
with some specific choice of $p(a,v)\in I$ \eqref{E:brycmbkn}. Under the current notational conventions, the kernel function may be written as $\braket{\phantom n|\phantom n}_p$, or just as $\braket{\phantom n|\phantom n}$ if the intended barycentric coordinate system $p$ is clear from the context.

With these conventions, the coordinate function $\Pi\to I$ of \eqref{E:CoordFun} appears as $\ket v_p$, or just $\ket v$, for each generator $v\in V$. Then
\begin{equation}\label{E:BraketP1}
\sum_{v\in V}\ket v=\mathds 1_\Pi
\end{equation}
expresses the partition of unity \eqref{E:PartoOne}, while
\begin{equation}\label{E:LinPrcsn}
\bra a=\sum_{v\in V}\braket{a|v}\bra v
\end{equation}
expresses the linear precision or barycentric property \eqref{E:brycmbkn}. In contrast with the partition of unity expression \eqref{E:BraketP1}, the expression
\begin{equation}\label{E:BraktPid}
\sum_{v\in V}\ket v\bra v=1_\Pi
\end{equation}
may be used to rewrite the barycentric property as a partition of the identity function $1_\Pi$ on the barycentric algebra $\Pi$.

\begin{remark}
The individual symbolism $\ket v\bra v$ appearing as a component of the sum on the left hand side of \eqref{E:BraktPid} does not have any independent meaning in the current context of barycentric algebras. However, if the barycentric algebra is a polytope $\Pi$ lying in a vector space $W$, with a chosen origin, then we could consider
$$
\ket v\bra v\colon\Pi\to W;a\mapsto p(a,v)v
$$
or $\bra a\mapsto\braket{a|v}\bra v$
as the function taking a point $a$ or $\bra a$ of $\Pi$ to the scalar multiple $p(a,v)v$ or $\braket{a|v}\bra v$ of the vector $v$ or $\bra v$ in $W$.
\end{remark}

In summary, the bra-ket notation provides a clear distinction between the two roles that may be played by a vertex $\mathbf v$, either appearing as a particular element $\bra{\mathbf v}$ of the polytope $\Pi$, or as supporting a coordinate function $\ket{\mathbf v}$ (or $\ket{\mathbf v}_p$) within the barycentric coordinate system $p$ on $\Pi$. Furthermore, the notation clearly separates the partition of $\mathds 1_\Pi$, as in \eqref{E:BraketP1}, from the partition of $1_\Pi$, as in \eqref{E:BraktPid}.

\subsection{Polytope coordinate systems}\label{SS:PolCoSys}

\begin{definition}\label{D:PolCoSys}
Let $\Pi$ be a polytope, with vertex set $V=\set{\mathbf v_1,\dots,\mathbf v_n}$. A \emph{coordinate system} for $\Pi$ is a map
\begin{equation}\label{E:CoSyKapa}
\lambda\colon V\to I^\Pi;\mathbf v\mapsto\ket{\mathbf v}_\lambda
\end{equation}
such that \eqref{E:BraktPid} holds.
\end{definition}

\begin{remark}\label{R:PolCoSys}
(a)
By Lemma~\ref{L:BPimpPO1}, the partition of unity property \eqref{E:BraketP1} holds in a coordinate system as specified by Definition~\ref{D:PolCoSys}.
\vskip 2mm
\noindent
(b)
The map $\lambda\colon V\to I^\Pi;\mathbf v\mapsto\ket{\mathbf v}_\lambda$ of \eqref{E:CoSyKapa} may be considered as an element of $(I^\Pi)^V$. Then, it may be curried to the parametrized family or map $\Pi\times V\to I;(\mathbf a,\mathbf v)\mapsto\braket{\mathbf a|\mathbf v}_\lambda$, an element of $I^{\Pi\times V}$. In the context of Theorem~\ref{T:PolCoSys} below, it will be convenient to use currying to identify the spaces $(I^\Pi)^V$ and $I^{\Pi\times V}$.
\end{remark}

\begin{theorem}\label{T:PolCoSys}
The set $K_\Pi$ of cooordinate systems on a polytope $\Pi$ with vertex set $V$ forms a convex subset of the hypercube $I^{\Pi\times V}$ under pointwise barycentric operations.
\end{theorem}

\begin{proof}
Using the notation of \eqref{E:CoSyKapa}, a coordinate system $\lambda$ is specified by the element $(\ket{\mathbf v_1}_\lambda,\dots,\ket{\mathbf v_n}_\lambda)$ of the pointwise barycentric algebra $(I^\Pi)^V$. Now, consider coordinate systems $\lambda,\lambda'$ and $p\in I^\circ$. By \eqref{E:BraktPid}, we have
\begin{equation}\label{E:lamblamb}
\sum_{\mathbf v\in V}\ket{\mathbf v}_{\lambda}\bra{\mathbf v}=1_\Pi
\quad
\mbox{and}
\quad
\sum_{\mathbf v\in V}\ket{\mathbf v}_{\lambda'}\bra{\mathbf v}=1_\Pi\,.
\end{equation}
The potential coordinate system $\lambda\lambda'\underline p$ is given by
$$
\ket{\mathbf v}_{\lambda\lambda'\underline p}
=\ket{\mathbf v}_{\lambda}\ket{\mathbf v}_{\lambda'}\underline p
$$
for $\mathbf{v}\in V$. Then
\begin{align*}
\sum_{\mathbf v\in V}\ket{\mathbf v}_{\lambda\lambda'\underline p}\bra{\mathbf v}
&
=
\sum_{\mathbf v\in V}\left(\ket{\mathbf v}_{\lambda}\ket{\mathbf v}_{\lambda'}\underline p
\right)\bra{\mathbf v}
=
\sum_{\mathbf v\in V}
\left(\ket{\mathbf v}_{\lambda}\bra{\mathbf v}\right)
\left(\ket{\mathbf v}_{\lambda'}\bra{\mathbf v}\right)
\underline p
\\
&
=
\left(
\sum_{\mathbf{v}\in V}\ket{\mathbf v}_{\lambda}\bra{\mathbf v}
\right)
\left(
\sum_{\mathbf{v}\in V}\ket{\mathbf v}_{\lambda'}\bra{\mathbf v}
\right)
\underline p
=1_\Pi1_\Pi\underline p
=1_\Pi
\end{align*}
using \eqref{E:lamblamb} for the penultimate equality. By Definition~\ref{T:PolCoSys}, it follows that $\lambda\lambda'\underline p$ is indeed a coordinate system.
\end{proof}

\begin{remark}\label{R:2ndaryPT}
Combinatorially, the convex set $K_\Pi$ of coordinate systems may be identified as the \emph{secondary polytope} of $\Pi$ \cite[p.7, \S5.1]{DLRS} or \cite[Ch.~7]{GelKapZel}. Thus, Theorem~\ref{T:PolCoSys} endows the secondary polytope with a rich geometric structure \cite{RJSAZD2}.
\end{remark}

\section{Chordal decompositions and coordinates}\label{S:ChDeCaCo}

In this chapter, we present an alternative and more local approach to the coordinates of polytopes, namely \emph{chordal coordinates}. The key idea is to decompose the polytope into simplices known as \emph{regions}, considering the resulting atlas of charts of the polytope, and directly taking the volumetric coordinates for the region within which a given point of the polytope lies.

In the interest of keeping the discussion relatively brief and accessible, and also to take advantage of special features that are available in the plane, we will restrict to the case of polygons.
Chordal coordinates are straightforward from the geometric point of view, but a certain measure of algebraic, topological and combinatorial sophistication is required for their efficient handling. In particular, we draw on the formalism established in the preceding chapter.

\subsection{Non-crossing chordal decomposition of polygons}

\subsubsection{Combinatorial and geometric vertex notation}\label{SSS:CombGeom}

Throughout this section, we will consider a convex polygon $\Pi$ presented as the convex hull of the counter-clockwise ordered sequence $\mathbf v_1,\dots,\mathbf v_n$ of extreme points (vertices) located around its boundary. It is often convenient to abbreviate the vertex set as the ordered set $V=\set{1<2<\ldots<n}$. The abbreviated notation refers to \emph{combinatorial} vertices, whereas the vertices in the full notation are described as \emph{geometric}.

\begin{definition}[Standard ordering]\label{D:OrdSbSet}
Consider a subset $S$ of $V$.
\begin{enumerate}
\item[$(\mathrm a)$]
If $|S|>2$, the \emph{standard order} on $S$ is its induced order.
\item[$(\mathrm b)$]
The \emph{standard order} on $S=\set{1,n}$ is $\set{n<1}$.
\item[$(\mathrm c)$]
If $S\ne\set{1,n}$ and $|S|=2$, then the \emph{standard order} on $S$ is its induced order.
\end{enumerate}
\end{definition}

\begin{lemma}\label{L:AdVxPair}
Reading vertex indices $i$ modulo $n$, the pairs $i<i+1$ are always in standard order.
\end{lemma}

For $j<k$ in the standard order (including $n<1$ and excluding $1<n$ by Definition~\ref{D:OrdSbSet}), the combinatorial notation $jk$ will stand for the vector $\mathbf v_k-\mathbf v_j$ or the (oriented) line segment $\mathbf v_{jk}$ from $\mathbf v_j$ to $\mathbf v_k$, and will be used to index functions or other constructs (such as the full oriented affine line $\mathcal L_{jk}$) that are determined by the oriented segment. The \emph{areal functions} of Definition~\ref{D:ArealFun} provide an important example.

\subsubsection{Skeleton graphs and chords}\label{SSS:SkGrChrd}

The \emph{skeleton} of the polygon $\Pi$ is the cyclic graph $C_n$ constituted by the vertices and undirected edges of the polygon. In the cyclic graph $C_n$, a \emph{chord} is an edge connecting vertices which are not adjacent in $C_n$.

\begin{definition}\label{D:SkGrChrd}
A \emph{chordal decomposition} of the polygon $\Pi$ with ordered vertex set $V=\set{\mathbf v_1<\mathbf v_2<\dots<\mathbf v_n}$ is a system of $n-3$ non-crossing chords of $C_n$ that decompose $\Pi$ as a union of $n-2$ simplices (triangles) whose vertices are vertices of $\Pi$. The $n-2$ triangles constitute the \emph{regions} of the decomposition. Each region (for $n>3$) is bounded by chords that form its \emph{internal boundaries}, and possibly (indeed, certainly for $n<6$) edges of the polygon that form its \emph{external boundaries}.
\end{definition}

Given any one chordal decomposition, we obtain others by the action of the dihedral group $D_n$ (of degree $n$ and order $2n$) as the automorphism group of the graph $C_n$ (cf. \cite{BowReg}). Leaving fixed the vertex set $V$ of the polygon $\Pi$, the elements of $D_n$ act on the $n-3$ chords of the decomposition. Explicit examples given in Remark~\ref{R:VerChTab}(c) provide a representative illustration of this action.

\subsubsection{Areal coordinate functions}\label{SSS:ArCordFun}

Consider a triangle in the Euclidean plane whose vertices are
$
\mathbf v_0=
\begin{bmatrix}
v_{00} &v_{01}
\end{bmatrix},
\mathbf v_1=
\begin{bmatrix}
v_{10} &v_{11}
\end{bmatrix}
\mbox{ and }
\mathbf v_2=
\begin{bmatrix}
v_{20} &v_{21}
\end{bmatrix}
$
in counterclockwise order. The formula for the signed area of this triangle is
\begin{equation}\label{E:Aov0v1v2}
A\left(\mathbf v_0,\mathbf v_1,\mathbf v_2\right)
=
\tfrac12\det[
\mathbf v_1-\mathbf v_0,
\mathbf v_2-\mathbf v_0
]
=
\frac12
\begin{vmatrix}
1 &v_{00} &v_{01}\\
1 &v_{10} &v_{11}\\
1 &v_{20} &v_{21}\\
\end{vmatrix}
\end{equation}

The following proposition recalls the convexity of the area $A\left(\mathbf v_0,\mathbf v_1,\mathbf v_2\right)$ with respect to variation of the vertex $\mathbf v_0$.

\begin{proposition}\label{P:Aov0v1v2}
Consider vectors
$$
\mathbf v_0=
\begin{bmatrix}
v_{00} &v_{01}
\end{bmatrix},
\mathbf v_0'=
\begin{bmatrix}
v_{00}' &v_{01}'
\end{bmatrix},
\mathbf v_1=
\begin{bmatrix}
v_{10} &v_{11}
\end{bmatrix}
\mbox{ and }
\mathbf v_2=
\begin{bmatrix}
v_{20} &v_{21}
\end{bmatrix}
$$
in the plane. Then
$$
A\big((1-p)\mathbf v_0+p\mathbf v_0',\mathbf v_1,\mathbf v_2\big)
=
(1-p)
A\left(\mathbf v_0,\mathbf v_1,\mathbf v_2\right)
+
p
A\left(\mathbf v_0',\mathbf v_1,\mathbf v_2\right)
$$
for each real number $p$.
\end{proposition}

\begin{proof}
The result follows by consideration of the Laplace expansion
$$
A\left(\mathbf v_0,\mathbf v_1,\mathbf v_2\right)=
\frac12
\left(
\begin{vmatrix}
v_{10} &v_{11}\\
v_{20} &v_{21}
\end{vmatrix}
-v_{00}
\begin{vmatrix}
1 &v_{11}\\
1 &v_{21}
\end{vmatrix}
+v_{01}
\begin{vmatrix}
1 &v_{10}\\
1 &v_{20}
\end{vmatrix}
\right)
$$
of the determinant of \eqref{E:Aov0v1v2} across its first row.
\end{proof}

\begin{proposition}\label{P:ArCordFn}
Let $\tau$ be a triangle whose vertices $\mathbf v_0<\mathbf v_1<\mathbf v_2<\mathbf v_0$ are oriented in cyclic counterclockwise order. For $\mathbf v_i\in\set{\mathbf v_0,\mathbf v_1,\mathbf v_2}$, the formula
\begin{equation}\label{E:ArCordFn}
\braket{\mathbf x|\mathbf v_i}_\tau=
\frac{A(\mathbf v_{i-1},\mathbf x,\mathbf v_{i+1})}{A(\mathbf v_0,\mathbf v_1,\mathbf v_2)}
\end{equation}
gives the areal coordinate of a point $\mathbf x$ of the triangle with respect to the vertex $\mathbf v_i$. The suffices are taken modulo 3.
\end{proposition}

\subsubsection{Areal functions}\label{SSS:ArealFun}

\begin{definition}\label{D:ArealFun}
Consider combinatorial vertices $j<k$ in standard order.
\begin{enumerate}
\item[$(\mathrm a)$]
The element
\begin{equation}\label{E:ArealFun}
\ket{j\wedge k}\colon \Pi\to\mathbb R;
\mathbf x\mapsto A(\mathbf x,\mathbf v_j,\mathbf v_k)
\end{equation}
of $\mathbb R^\Pi$ is the \emph{areal function} defined by the pair $j<k$.
\item[$(\mathrm b)$]
The element
\begin{equation*}
\ket{k\wedge j}\colon \Pi\to\mathbb R;
\mathbf x\mapsto A(\mathbf x,\mathbf v_k,\mathbf v_j)
=-A(\mathbf x,\mathbf v_j,\mathbf v_k)
\end{equation*}
of $\mathbb R^\Pi$ is the \emph{reverse areal function} defined by the pair $j<k$.
\item[$(\mathrm c)$]
If $\mathbf v_{jk}$ is an edge of the polygon, the areal function $\ket{j\wedge k}$ is described as an \emph{external boundary function}.
\item[$(\mathrm d)$]
If $\mathbf v_{jk}$ is a chord of a given chordal decomposition of the polygon, the areal function $\ket{j\wedge k}$ is described as an \emph{internal boundary function}.
\end{enumerate}
\end{definition}

\begin{lemma}\label{L:ArFuBaHo}
By Proposition~\ref{P:Aov0v1v2}, the areal function \eqref{E:ArealFun} is a barycentric algebra homomorphism from the convex set $\Pi$ to the convex set $\mathbb R$.
\end{lemma}

\begin{lemma}\label{L:ArealFun}
Fix a pair $j<k$ of combinatorial vertices in standard order, and an element $\mathbf a$ of $\Pi$. Consider the line $\mathcal L_{jk}$ joining the vectors $\mathbf v_j$ and $\mathbf v_k$, traversed in the direction from $\mathbf v_j$ to $\mathbf v_k$.
\begin{enumerate}
\item[$(\mathrm a)$]
If $\mathbf a$ lies on $\mathcal L_{jk}$, then $\braket{\mathbf a|j\wedge k}=0$.
\item[$(\mathrm b)$]
If $\mathbf a$ lies to the left of $\mathcal L_{jk}$, then $\braket{\mathbf a|j\wedge k}>0$.
\item[$(\mathrm c)$]
If $\mathbf a$ lies to the right of $\mathcal L_{jk}$, then $\braket{\mathbf a|j\wedge k}<0$.
\end{enumerate}
\end{lemma}

\begin{lemma}\label{L:aiiplus1}
Consider an element $\mathbf a$ of $\Pi$.
\begin{enumerate}
\item[$(\mathrm a)$]
For $1\le i\le n$, we have $\braket{\mathbf a|i\wedge(i+1)}\ge0$ for the external boundary functions specified using the convention of Lemma~\ref{L:AdVxPair}.
\item[$(\mathrm b)$]
We have $\braket{\mathbf a|i\wedge(i+1)}>0$ for all $1\le i\le n$ if and only if $\mathbf a$ lies in the interior of $\Pi$.
\item[$(\mathrm c)$]
Suppose that $\mathbf a$ is a combinatorial vertex $j$ with $j\notin\set{i,(i+1)}$. Then $\braket{j|i\wedge(i+1)}>0$.
\end{enumerate}
\end{lemma}

\subsubsection{The hexagon as a representative example}

Figure~\ref{F:HexDcomp} displays three distinct decompositions of a hexagon into $4$ triangles by means of three non-crossing chords. The three topologically distinct types provide a full set of representatives for the orbits of the dihedral group $D_6$ on the chordally subdivided graph $C_6$. The chords are depicted by thin lines, and the edges of the hexagon by thick lines. The boxed labels, which distinguish the three topological types, indicate the (\emph{chordal}) \emph{degree sequence} (CDS) of the graph formed by the chords.
For example, the chord graph in the middle case has $\mathbf v_1$ and $\mathbf v_4$ as vertices of degree one, together with $\mathbf v_3$ and $\mathbf v_6$ as vertices of degree two.

The examples shown in Figure~\ref{F:HexDcomp} represent the general situation. Indeed, $6$ is the smallest number of polygon vertices capable of exhibiting all the potentialities: for $2<n<6$, the only possible chordal degree sequence is $1^{(n-3)}(n-3)$. Note that the region $\mathbf v_1\mathbf v_3\mathbf v_5$ of the third hexagon ($2^3$) has no external boundaries. Given the general nature of the examples from Figure~\ref{F:HexDcomp}, their subsequent discussion will suffice to indicate the analysis of a general polygon.

\begin{figure}[hbt]
  \centering
  \begin{tikzpicture}
    \coordinate (Origin)   at (0,0);
    \coordinate (XAxisMin) at (-1,0);
    \coordinate (XAxisMax) at (10.5,0);
    \coordinate (YAxisMin) at (0,-1);
    \coordinate (YAxisMax) at (0,3.5);
        \clip (-1,-1) rectangle (11cm,3.5cm);
\node[label=above:
{
$
\boxed{1^33}
$
}
]
at (0.9,2.5) {};
\draw[thick] (0,0) -- (0,1);
\draw[thick,<-] (0,0.4) -- (0,1);
\draw[thick] (0,0) -- (1,0);
\draw[thick,->] (0,0) -- (0.6,0);
\draw[thick] (1,0) -- (2,1);
\draw[thick,->] (1,0) -- (1.6,0.6);
\draw[thick] (1,2) -- (2,2);
\draw[thick,<-] (1.4,2) -- (2,2);
\draw[thick] (2,1) -- (2,2);
\draw[thick,->] (2,1) -- (2,1.6);
\draw[thick] (0,1) -- (1,2);
\draw[thick,<-] (0.4,1.4) -- (1,2);
\node[label=right:{$\mathbf v_1$}, draw, circle, inner sep=2pt, fill] at (2,1) {};
\node[label=above:{$\mathbf v_2$}, draw, circle, inner sep=2pt, fill] at (2,2) {};
\node[label=above:{$\mathbf v_3$}, draw, circle, inner sep=2pt, fill] at (1,2) {};
\node[label=left:{$\mathbf v_4$}, draw, circle, inner sep=2pt, fill] at (0,1) {};
\node[label=below:{$\mathbf v_5$}, draw, circle, inner sep=2pt, fill] at (0,0) {};
\node[label=below:{$\mathbf v_6$}, draw, circle, inner sep=2pt, fill] at (1,0) {};
\draw (2,1) -- (1,2);
\draw[->] (2,1) -- (1.45,1.55);
\draw (2,1) -- (0,1);
\draw[->] (2,1) -- (0.95,1);
\draw (2,1) -- (0,0);
\draw[->] (2,1) -- (0.9,0.45);
\node[label=above:
{
$
\boxed{1^22^2}
$
}
]
at (4.9,2.5) {};
\draw[thick] (4,0) -- (4,1);
\draw[thick,<-] (4,0.4) -- (4,1);
\draw[thick] (4,0) -- (5,0);
\draw[thick,->] (4,0) -- (4.6,0);
\draw[thick] (5,0) -- (6,1);
\draw[thick,->] (5,0) -- (5.6,0.6);
\draw[thick] (5,2) -- (6,2);
\draw[thick,<-] (5.4,2) -- (6,2);
\draw[thick] (6,1) -- (6,2);
\draw[thick,->] (6,1) -- (6,1.6);
\draw[thick] (4,1) -- (5,2);
\draw[thick,<-] (4.4,1.4) -- (5,2);
\node[label=right:{$\mathbf v_1$}, draw, circle, inner sep=2pt, fill] at (6,1) {};
\node[label=above:{$\mathbf v_2$}, draw, circle, inner sep=2pt, fill] at (6,2) {};
\node[label=above:{$\mathbf v_3$}, draw, circle, inner sep=2pt, fill] at (5,2) {};
\node[label=left:{$\mathbf v_4$}, draw, circle, inner sep=2pt, fill] at (4,1) {};
\node[label=below:{$\mathbf v_5$}, draw, circle, inner sep=2pt, fill] at (4,0) {};
\node[label=below:{$\mathbf v_6$}, draw, circle, inner sep=2pt, fill] at (5,0) {};
\draw (6,1) -- (5,2);
\draw[->] (6,1) -- (5.45,1.55);
\draw (5,2) -- (5,0);
\draw[->] (5,2) -- (5,0.9);
\draw (4,1) -- (5,0);
\draw[->] (4,1) -- (4.55,0.45);
\node[label=above:
{
$
\boxed{2^3}
$
}
]
at (8.9,2.5) {};
\draw[thick] (8,0) -- (8,1);
\draw[thick,<-] (8,0.4) -- (8,1);
\draw[thick] (8,0) -- (9,0);
\draw[thick,->] (8,0) -- (8.6,0);
\draw[thick] (9,0) -- (10,1);
\draw[thick,->] (9,0) -- (9.6,0.6);
\draw[thick] (9,2) -- (10,2);
\draw[thick,<-] (9.4,2) -- (10,2);
\draw[thick] (10,1) -- (10,2);
\draw[thick,->] (10,1) -- (10,1.6);
\draw[thick] (8,1) -- (9,2);
\draw[thick,<-] (8.4,1.4) -- (9,2);
\node[label=right:{$\mathbf v_1$}, draw, circle, inner sep=2pt, fill] at (10,1) {};
\node[label=above:{$\mathbf v_2$}, draw, circle, inner sep=2pt, fill] at (10,2) {};
\node[label=above:{$\mathbf v_3$}, draw, circle, inner sep=2pt, fill] at (9,2) {};
\node[label=left:{$\mathbf v_4$}, draw, circle, inner sep=2pt, fill] at (8,1) {};
\node[label=below:{$\mathbf v_5$}, draw, circle, inner sep=2pt, fill] at (8,0) {};
\node[label=below:{$\mathbf v_6$}, draw, circle, inner sep=2pt, fill] at (9,0) {};
\draw (10,1) -- (9,2);
\draw[->] (10,1) -- (9.45,1.55);
\draw (9,2) -- (8,0);
\draw[->] (9,2) -- (8.45,0.9);
\draw (8,0) -- (10,1);
\draw[<-] (8.9,0.45) -- (10,1);
\end{tikzpicture}
\caption{Oriented, non-crossing chordal decompositions of hexagons.}
\label{F:HexDcomp}
\end{figure}
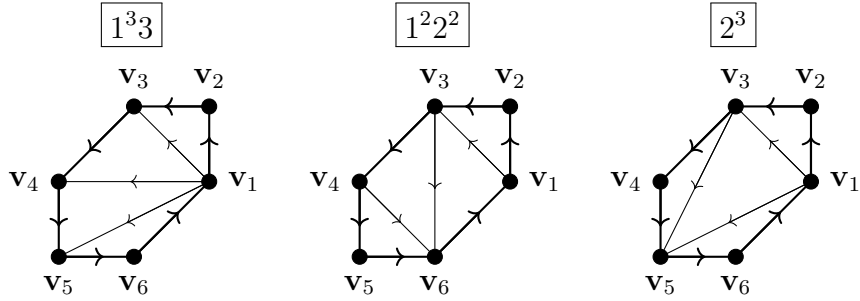

\renewcommand{\arraystretch}{1.5}
\begin{table}[hbt]
\begin{tabular}{||c||c|c|c||c|c|c||c|c|c||}
\hline
CDS:
&\multicolumn{3}{|c||}{$1^33$}
&\multicolumn{3}{|c||}{$1^22^2$}
&\multicolumn{3}{|c||}{$2^3$}
\\
\hline
Chord:
 &$\ket{13}$ &$\ket{14}$ &$\ket{15}$
 &$\ket{13}$ &$\ket{36}$ &$\ket{46}$
 &$\ket{13}$ &$\ket{15}$ &$\ket{35}$
\\
\hline
$\bra 1$
&$0$ &$0$ &$0$
&$0$ &$+$ &$+$
&$0$ &$0$ &$+$
\\
\hline
$\bra 2$
&$-$ &$-$ &$-$
&$-$ &$+$ &$+$
&$-$ &$-$ &$+$
\\
\hline
$\bra 3$
&$0$ &$-$ &$-$
&$0$ &$0$ &$+$
&$0$ &$-$ &$0$
\\
\hline
$\bra 4$
&$+$ &$0$ &$-$
&$+$ &$-$ &$0$
&$+$ &$-$ &$-$
\\
\hline
$\bra 5$
&$+$ &$+$ &$0$
&$+$ &$-$ &$-$
&$+$ &$0$ &$0$
\\
\hline
$\bra 6$
&$+$ &$+$ &$+$
&$+$ &$0$ &$0$
&$+$ &$+$ &$+$
\\
\hline
\end{tabular}
\vskip 2mm
\caption{The respective \emph{vertex/chord tables} for the hexagon decompositions depicted in Figure~\ref{F:HexDcomp}. The second row of the vertex/chord table labels the table columns that correspond to the three internal boundary functions, for each chord, for each of the three chordal decompositions labeled in the first table row by its chordal degree sequence. Then, for each combinatorial vertex $i$ of $A$, represented in the table as the bra row label $\bra i$, the table entry in the column $\ket{jk}$ is the sign of the areal function value $\braket{i|j\wedge k}$ determined by \eqref{E:ArealFun}.}
\label{T:vertchrd}
\end{table}

\subsubsection{Vertex/chord tables}\label{SSS:VerChTab}

\begin{remark}\label{R:VerChTab}
Consider Table~\ref{T:vertchrd}.
\vskip 2mm
\noindent
(a)
In the table, the fuller column labels $\ket{j\wedge k}$ would be more appropriate that $\ket{jk}$. The extra $\wedge$ symbols have been omitted for typographical reasons.
\vskip 2mm
\noindent
(b)
The CDS labels in the first row of the table are generic. More explicit specifications of the actual chord decompositions being used could appear as
\begin{equation}\label{E:ChDunits}
\mathds 1_{13}\circ\mathds 1_{14}\circ\mathds 1_{15}\,,\
\mathds 1_{13}\circ\mathds 1_{36}\circ\mathds 1_{46}\,,\
\mathds 1_{13}\circ\mathds 1_{15}\circ\mathds 1_{35}
\end{equation}
using the combinatorial notations $jk$ for the individual chords, and the Boolean notation of Lemma~\ref{L:IdInRtoB}(b) within the algebra $\mathbb R^\Pi$ to represent the specific chords that constitute the decomposition. In particular, the CDS labels may be recovered from \eqref{E:ChDunits}.
\vskip 2mm
\noindent
(c)
The middle chordal decomposition of Figure~\ref{F:HexDcomp} and \eqref{E:ChDunits} may be used to illustrate the dihedral group action on chordal decompositions discussed in \S\ref{SSS:SkGrChrd}. The dihedral group $D_6$ is generated by the \emph{rotation} $\rho=(1\ 2\ 3\ 4\ 5\ 6)$, addition of $1$ modulo $6$, and the \emph{reflection} (German: Spiegelung) $\sigma=(1\ 5)(2\ 4)$, negation modulo $6$. Using the notation of (b), the respective rotation and reflection actions such as
\begin{align*}
&
\mathds 1_{13}\circ\mathds 1_{36}\circ\mathds 1_{46}
\overset{\rho}{\mapsto}
\mathds 1_{24}\circ\mathds 1_{41}\circ\mathds 1_{51}
=
\mathds 1_{14}\circ\mathds 1_{15}\circ\mathds 1_{24}
\
\mbox{ and}
\\
&
\mathds 1_{13}\circ\mathds 1_{36}\circ\mathds 1_{46}
\overset{\sigma}{\mapsto}
\mathds 1_{53}\circ\mathds 1_{36}\circ\mathds 1_{26}
=
\mathds 1_{26}\circ\mathds 1_{35}\circ\mathds 1_{36}
\end{align*}
then suffice to specify the right $D_6$ action on the complete set of all $6$ chordal decompositions with CDS $1^22^2$.
\end{remark}

Note that the table entries $\braket{j|jk}$ and $\braket{k|jk}$ are zero. For a geometric vertex $\mathbf v_i$ not lying on the line $\mathcal L_{jk}$ from $\mathbf v_j$ to $\mathbf v_k$, imagine following that line in the direction of the arrow, as described in Lemma~\ref{L:ArealFun}. Then if the vertex $\mathbf v_i$ lies to the left, the table entry for $\braket{i|jk}$ is $+$. If the vertex is on the right, the table entry is $-$. For example, consider the vertex $\mathbf v_4$ in the middle hexagon of Figure~\ref{F:HexDcomp}. Passing from $\mathbf v_1$ to $\mathbf v_3$, it lies on the left, while passing from $\mathbf v_3$ to $\mathbf v_6$, it lies on the right. Finally, it lies on the line from $\mathbf v_4$ to $\mathbf v_6$. This accounts for the consecutive entries $+,-,0$ in the middle part of the table row that is labelled by $\bra 4$.

\renewcommand{\arraystretch}{1}

\subsection{Region identification and orientation}\label{SSRegnIdOr}

A chordal decomposition $\delta$ of the polygon $\Pi$ (non-trivially for $n>3$, with $n=3$ as a base case) decomposes it as a union of $n-2$ triangular regions $\tau$. These triangular regions intersect with their neighbors along internally bounding chords. Let $T_\delta$ denote the set of $n-2$ triangles $\tau$ appearing in $\delta$. First, the triangular regions $\tau$ must be identified. Then, for each triangle $\tau$ in $T_\delta$, its vertex set $\set{\mathbf v_{i_1},\mathbf v_{i_2},\mathbf v_{i_3}}$ must be arranged in cyclic counterclockwise order $\mathbf v_{i_1}\wedge \mathbf v_{i_2}\wedge \mathbf v_{i_3}$ to facilitate the application of Proposition~\ref{P:ArCordFn}.

\begin{remark}
Our use of exterior algebra notation, here $\mathbf v_{i_1}\wedge \mathbf v_{i_2}\wedge \mathbf v_{i_3}$ for the oriented simplex $\tau$, or indeed $j\wedge k$ for an oriented line segment in Definition~\ref{D:ArealFun}, follows \cite[\S IV.4]{Chevalley1}, as reprinted as \cite[p.62]{Chevalley2}.
\end{remark}

We outline a recursive procedure for region identification and orientation, including the location of a given point within a region. The procedure is ultimately founded on the correspondence between chordal decompositions of polygons and rooted binary trees.

\subsubsection{Chordal decompositions and non-coassociative coproduct parsing trees}\label{SSS:ChDePaTr}

While rooted binary trees that correspond to chordal decompositions are usually considered to be parsing trees of repeated non-associative products, it is actually more appropriate to interpret them as parsing trees of repeated non-coassociative coproducts $\Delta\colon x\mapsto x^L\otimes x^R$:
\begin{equation}\label{E:coprodct}
\xymatrix{
&
\mathbf v_{i_3}
\ar
_-{x^L}[ddl]
\\
\\
\mathbf v_{i_1}
\ar
_-{x}[rr]
&
&
\mathbf v_{i_2}
\ar
_-{x^R}[uul]
}
\end{equation}
Here, we are using the non-coassociative Sweedler notation of \cite[Remark 2.2(b)]{Sm171}. For a concrete interpretation, consider the coproduct as acting on a probability distribution $x$ on the interior segment from $\mathbf v_{i_1}$ to $\mathbf v_{i_2}$ in \eqref{E:coprodct}. By some specified transport mechanism, it then produces a probability distribution $x^L$ on the interior segment from $\mathbf v_{i_3}$ to $\mathbf v_{i_1}$, and a probability distribution $x^R$ on the interior segment from $\mathbf v_{i_2}$ to $\mathbf v_{i_3}$. We denote the oriented segment $\mathbf v_{i_1}\wedge \mathbf v_{i_2}$ as the \emph{base} or \emph{input}, the oriented segment $\mathbf v_{i_2}\wedge \mathbf v_{i_3}$ as the \emph{right output}, and the oriented segment $\mathbf v_{i_3}\wedge \mathbf v_{i_1}$ as the \emph{left output}.

\subsubsection{The recursive step}

The input to the recursive step is a triple $(P,W,\chi)$ consisting of a polygon $P$ spanned by a vertex set $W=\set{\mathbf w_1<\dots<\mathbf w_r}$ that supports a standard ordering, and a set $\chi$ of standardly ordered chords that decompose $P$. The set $\chi$ will be empty if $r<4$.

The recursive step outputs a pair $(P^L,W^L,\chi^L)$,  $(P^R,W^R,\chi^R)$ of such triples. Either one of these output triples may be \emph{terminal}, in the sense that its vertex set contains only two elements. The input triple is not allowed to be terminal.

Given the input triple $(P,W,\chi)$, the \emph{base} is the oriented boundary edge $\mathbf w_1\wedge \mathbf w_2$, or just $1\wedge 2$ in the combinatorial notation. If $r=3$, then $1\wedge 2\wedge 3$ is the only oriented region $\tau$, and the procedure halts, since both of its output triples are terminal: the left $W^L=\set{3<1}$ and the right $W^R=\set{2<3}$. We variously identify
$$
\tau(P,W,\chi)=\tau_{1\wedge2}=1\wedge2\wedge3
$$
as the triangular region \emph{selected} by $(P,W,\chi)$.

Otherwise, if $r>3$, the triangular region $\tau$ including the edge $1\wedge 2$ takes the oriented form $1\wedge 2\wedge d$ for some vertex $d$. Again, we identify
\begin{equation}\label{E:regselct}
\tau(P,W,\chi)=\tau_{1\wedge2}=1\wedge2\wedge d
\end{equation}
as the triangular region \emph{selected} by $(P,W,\chi)$. With geometric vertex labels, the orientations are illustrated by
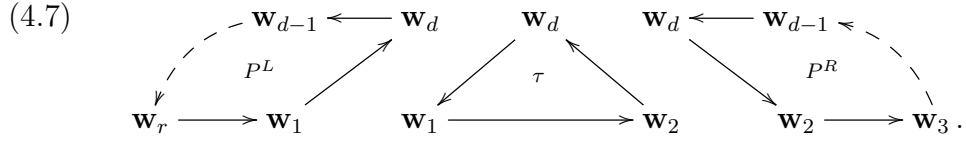
\begin{equation}\label{E:ltrtpoly}
\xymatrix{
&
\mathbf w_{d-1}
\ar@/_1pc/@{-->}[dl]
\ar@{}[d]_{P^L}
&
\mathbf w_d
\ar[l]
&
\mathbf w_d
\ar[dl]
\ar@{}@<1ex>[d]_-{\tau}
&
\mathbf w_d
\ar[dr]
&
\mathbf w_{d-1}
\ar[l]
\ar@{}[d]^{P^R}
\\
\mathbf w_r
\ar[r]
&
\mathbf w_1
\ar[ur]
&
\mathbf w_1
\ar[rr]
&
&
\mathbf w_2
\ar[ul]
&
\mathbf w_2
\ar[r]
&
\mathbf w_3\,.
\ar@/_1pc/@{-->}[ul]
}
\end{equation}

The vertex $d$ is determined from the chordal decomposition $\chi$ as follows:
\begin{enumerate}
\item[$(\mathrm a)$]
In this first case, both $1\wedge d$ and $2\wedge d$ are chords, so $\boxed{3<d<r}$.
\item[$(\mathrm b^L)$]
If $2\wedge d$ is a chord, but $1\wedge d$ is not, then $\boxed{d=r}$. The vertex set $\set{1,r}$ of $1\wedge r$ appears as the oriented boundary edge $r\wedge 1$;
\item[$(\mathrm b^R)$]
If $1\wedge d$ is a chord, but $2\wedge d$ is not, then $2\wedge d$ is the boundary edge $2\wedge 3$. In particular, $\boxed{d=3}$.
\end{enumerate}

The case designations $\mathrm b^L$ or $\mathrm b^R$ respectively signify that the carrier of $x^L$ or $x^R$ in the coproduct is a boundary edge. The recursion step behaves somewhat differently in each of the three cases:
\begin{enumerate}
\item[$(\mathrm a)$]
In this first case, $d\wedge 2$ assumes the role of the base boundary edge for the subpolygon $P^R$ with vertex set $W^R=\set{d<2<3<\dots<d-1}$. Here $\chi^R=\chi\cap\,\mathrm{int}\left(P^R\right)$. Similarly, the oriented edge $1\wedge d$ assumes the role of the base boundary edge for the subpolygon $P^L$ with vertex set $W^L=\set{1<d<d+1<\dots<r}$, and $\chi^L=\chi\cap\,\mathrm{int}\left(P^L\right)$.
\item[$(\mathrm b^L)$]
Here, $r\wedge 1$ plays no further role, corresponding to a leaf in the parsing tree. Formally, $W^L=\set{r<1}$ is terminal. On the other hand, $r\wedge 2$ now assumes the role of the base boundary edge for the subpolygon $P^R$ with vertex set $W^R=\set{r<2<\dots<r-1}$, and $\chi^R=\chi\cap\,\mathrm{int}\left(P^R\right)$.
\item[$(\mathrm b^R)$]
Here, $W^R=\set{2<3}$ is terminal. Now, $1\wedge 3$ assumes the role of the base boundary edge for the subpolygon $P^L$ with vertex set $W^L=\set{1<3<\dots<r}$, and $\chi^L=\chi\cap\,\mathrm{int}\left(P^L\right)$.
\end{enumerate}

The recursive step identifies the diagram
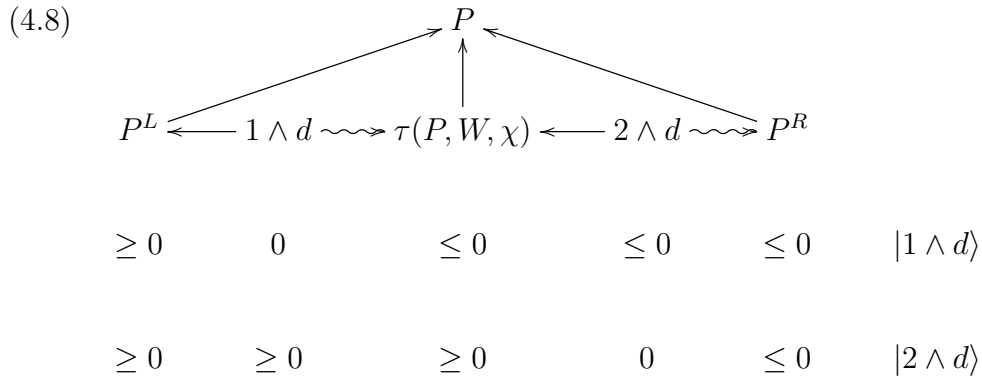
\begin{equation}\label{E:rechordt}
\xymatrix{
&
&
P
\\
P^L
\ar[urr]
&
1\wedge d
\ar[l]
\ar@{~>}[r]
&
\tau(P,W,\chi)
\ar[u]
&
2\wedge d
\ar[l]
\ar@{~>}[r]
&
P^R
\ar[ull]
\\
\ge0
&
0
&
\le0
&
\le0
&
\le0
&
\ket{1\wedge d}
\\
\ge0
&
\ge0
&
\ge0
&
0
&
\le0
&
\ket{2\wedge d}
}
\end{equation}
whose upper part exhibits unoriented subset relationships. In addition, the upper part may be interpreted to give oriented subset relationships, with the convention that wavy arrows reverse the orientation during the inclusion. The orientations correspond to those depicted in \eqref{E:ltrtpoly}. In case (a), all the depicted inclusions are proper. In case $(\mathrm b^L)$, the inclusion of $1\wedge d$ in $P^L$ is improper. In case $(\mathrm b^R)$, the orientation-reversing inclusion of $2\wedge d$ in $P^R$ is improper.

The bottom rows of \eqref{E:rechordt} give a fragment of an extended vertex/chord table. The columns are labeled by the various sets $S$ in the second row of the diagram. For an element $\mathbf x$ of $S$, the entry in the row labeled by $\ket{i\wedge d}$ gives sign information for $\braket{\mathbf x|i\wedge d}$. Furthermore, the listed behavior of $\mathbf x\in S=\tau(P,W,\chi)=\tau(\Pi^\Lambda,V^\Lambda,\delta^\Lambda)$ applies equally to the behavior of any element $\mathbf y$ of $\tau(\Pi^{\Lambda_I},V^{\Lambda_I},\delta^{\Lambda_I})$ for any initial segment $\Lambda_I$ of $\Lambda$, since those elements of $\Pi$ appear in the convex set $\Pi\smallsetminus(P^L\cup P^R)$ which includes $\mathbf w_1\wedge \mathbf w_2$ as an edge or chord.

The role of the bottom rows of \eqref{E:rechordt} in the location of a polygon element $\mathbf x$ is illustrated by the example of \S\ref{SSS:hexample}.

\subsubsection{The recursive procedure}

The recursive step is applied $n-2$ times in a complete determination of the regions in the chordal decomposition. At the initial step, we take $(P,W,\chi)$ as $(\Pi,V,\delta)$, with base $1\wedge 2$, selecting the triangular region $\tau^{\varnothing}=\tau(\Pi,V,\delta)$. Subsequent recursive steps work with
$$
(P,W,\chi)=(\Pi^\Lambda,V^\Lambda,\delta^\Lambda)\,,
$$
where $\Lambda$ is a word in the alphabet $\set{L,R}$. Indeed, we may also write $(\Pi^{\varnothing},V^{\varnothing},\delta^{\varnothing})$ for the initial triple $(\Pi,V,\delta)$, using the empty word $\varnothing$ in the alphabet $\set{L,R}$. The triangular region selected from $(\Pi^\Lambda,V^\Lambda,\delta^\Lambda)$ may be denoted by $\tau^\Lambda$, in addition to the notations of \eqref{E:regselct}. These conventions, and further fine details of the recursive procedure, are adequately illustrated on the basis of the following example.

\subsubsection{An example from the hexagon}\label{SSS:hexample}

In our hexagonal example with CDS $2^3$, the coproduct parsing tree may be displayed as follows:
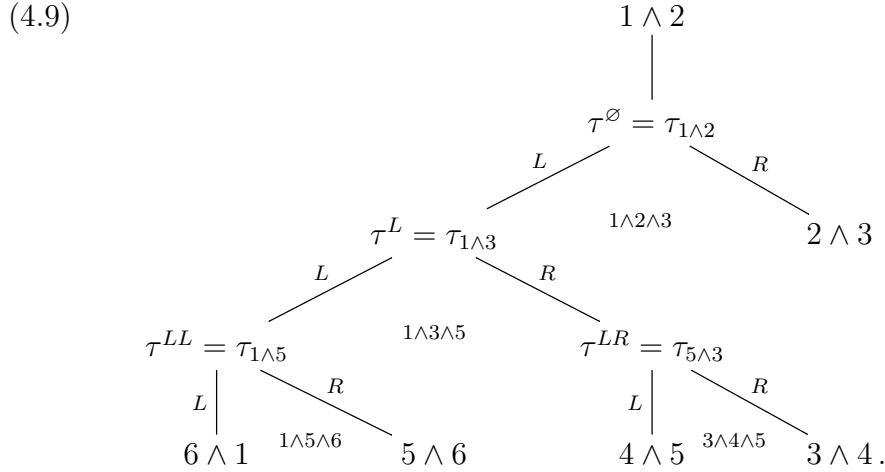
\begin{equation}\label{23PrsTre}
\xymatrix{
&
&
1\wedge 2
\ar@{-}[d]
\\
&
&
\tau^{\varnothing}=\tau_{1\wedge 2}
\ar@{-}[dl]_-{L}
\ar@{-}[dr]^-{R}
\\
&
\tau^{L}=\tau_{1\wedge 3}
\ar@{-}[dl]_-{L}
\ar@{-}[dr]^-{R}
\ar@{}[rr]^{1\wedge2\wedge3}
&
&
2\wedge 3\phantom{\,.}
\\
\tau^{LL}=\tau_{1\wedge 5}
\ar@{-}[d]_-{L}
\ar@{-}[dr]^-{R}
\ar@{}[rr]^{1\wedge3\wedge5}
&
&
\tau^{LR}=\tau_{5\wedge 3}
\ar@{-}[d]_-{L}
\ar@{-}[dr]^-{R}
\\
6\wedge 1
\ar@{}[r]^{1\wedge5\wedge6\rule{4mm}{0mm}}
&
5\wedge 6
&
4\wedge 5
\ar@{}[r]^{3\wedge4\wedge5\rule{4mm}{0mm}}
&
3\wedge 4\,.
}
\end{equation}
The initial base $1\wedge 2$, a boundary edge, is presented as an initial leaf. The remaining $5$ boundary edges appear as the remaining leaves, marking respective terminations of branches of the recursive procedure. The initial triangle is labelled by the initial base, while the remaining internal nodes mark triangles labelled by (possibly reoriented) chords. The reorientations are determined by the wavy arrows in the second row of the diagram \eqref{E:rechordt}. The standard orderings of the triangles are also placed within the branchings of the tree, showing how they comprise the three surrounding edges and chords from the branching.

At the first internal node $\tau^{\varnothing}$ of the tree, with $(P,W,\chi)=(\Pi,V,\delta)$, the middle rows of \eqref{E:rechordt} take the form
\begin{equation}\label{E:onewedge}
\scalebox{.96}{
\xymatrix{
\Pi^L
&
1\wedge 3
\ar[l]
\ar@{~>}[r]
&
\tau(\Pi,V,\delta)
&
2\wedge 3
\ar[l]
\ar@{~>}[r]
&
\Pi^R
\\
\ge0
&
0
&
\le0
&
\le0
&
\le0
&
\ket{1\wedge 3}.
}
}
\end{equation}
The final row is not recorded here, since it displays an external boundary function which always returns nonnegative values at any element $\mathbf x$ of $\Pi$. At the second internal node $\tau^L$ of the tree, with $W=\set{1<3<4<5<6}$ and $\chi=\set{1\wedge 5,3\wedge5}$, the lower rows of \eqref{E:rechordt} take the form
\begin{equation}\label{E:twowedge}
\scalebox{.96}{
\xymatrix{
\Pi^{LL}
&
1\wedge 5
\ar[l]
\ar@{~>}[r]
&
\tau(\Pi^L,W,\chi)
&
3\wedge 5
\ar[l]
\ar@{~>}[r]
&
\Pi^{LR}
\\
\ge0
&
0
&
\le0
&
\le0
&
\le0
&
\ket{1\wedge 5}\phantom{.}
\\
\ge0
&
\ge0
&
\ge0
&
0
&
\le0
&
\ket{3\wedge 5}.
}
}
\end{equation}

The information from the lower rows of \eqref{E:onewedge} and \eqref{E:twowedge} is collated in the columns of the \emph{triangle/chord table}, displayed as Table~\ref{T:VerChTab}.

\renewcommand{\arraystretch}{1.5}
\begin{table}[hbt]
\begin{tabular}{||r||c|c|c||c||}
\hline
Chord: &$\ket{1\wedge3}$ &$\ket{1\wedge5}$ &$\ket{3\wedge5}$ &Code\\
\hline
$\tau^{\varnothing}$ &$\le0$ &$\le0$ &$\ge0$ &001\\
\hline
$\Pi^L\supset\tau^{L}$ &$\ge0$ &$\le0$ &$\ge0$ &101\\
\hline
$\Pi^L\supset\Pi^{LL}=\tau^{LL}$ &$\ge0$ &$\ge0$ &$\ge0$ &111\\
\hline
$\Pi^L\supset\Pi^{LR}=\tau^{LR}$ &$\ge0$ &$\le0$ &$\le0$ &100\\
\hline
\end{tabular}
\vskip 4mm
\caption{The \emph{triangle/chord table}, which is structured like the vertex/chord Table~\ref{T:vertchrd}. Its initial columns incorporate the information from the lower rows of \eqref{E:onewedge} and \eqref{E:twowedge}. Then, the binary coding in the final column records $0$ for a non-positive chordal areal function value, and $1$ for a non-negative chordal areal function value. Each triangle has a distinct coding, enabling location of any polygon point $\mathbf x$ within its triangular region(s).}
\label{T:VerChTab}
\end{table}

\begin{remark}\label{R:Hodgdual}
The parsing tree displayed in \eqref{23PrsTre} is topologically dual to the triangularly partitioned hexagon, as shown by the following diagram.
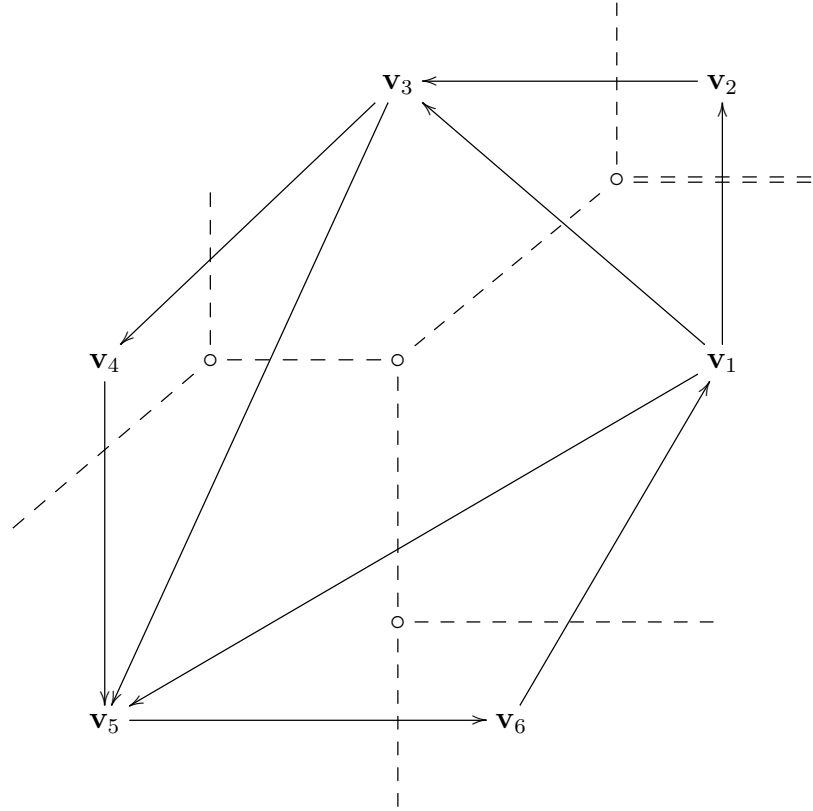
\begin{equation}\label{E:Hodgdual}
\xymatrix{
&
&
&
&
&
&
\ar@{--}[dd]
\\
&
&
&
&
\mathbf v_3
\ar[dddlll]
\ar[dddddddlll]
&
&
&
\mathbf v_2
\ar[lll]
\\
&
&
\ar@{--}[dd]
&
&
&
&
\circ
\ar@{==}[rr]
&
&
\\
\\
&
\mathbf v_4
\ar[dddd]
&
\circ
\ar@{--}[rr]
&
&
\circ
\ar@{--}[uurr]
&
&
&
\mathbf v_1
\ar[uuu]
\ar[uuulll]
\ar[ddddllllll]
\\
\\
\ar@{--}[uurr]
\\
&
&
&
&
\circ
\ar@{--}[uuu]
\ar@{--}[rrr]
\ar@{--}[dd]
&
&
&
\\
&
\mathbf v_5
\ar[rrrr]
&
&
&
&
\mathbf v_6
\ar[uuuurr]
\\
&
&
&
&
}
\end{equation}
The open circles represent $0$-dimensional points dual to the $2$-dimensional triangles, while the dashed lines represent $1$-dimensional edges that are dual to the $1$-dimensional chords and hexagon edges. The double dashed edge, emerging from the root of the parsing tree, is dual to the base edge $1\wedge 2$ of the hexagon.

Some diagrams similar to \eqref{E:Hodgdual} were exhibited in an abstract context in \cite[Fig.~4.15]{ConwayGuy}. Now, it is interesting to note how the topological duality of the diagrams tracks the duality between coalgebras and algebras.
\end{remark}

\subsubsection{Some sample computations}\label{SSS:SoSamCom}

In the Euclidean plane $\mathbb R^2$,  consider the hexagon $\Pi$ spanned by the vertices
$$
\mathbf v_1 = (2,1),\ \mathbf v_2 = (2,2),\ \mathbf v_3 = (1,2),\
\mathbf v_4 = (0,1),\ \mathbf v_5 = (0,0),\ \mathbf v_6 = (1,0)
$$
with the chordal decomposition $\delta=\mathds 1_{13}\circ\mathds 1_{15}\circ\mathds 1_{35}$ of type CDS $2^3$ from \eqref{E:ChDunits}. We determine the chordal coordinates of the three points
$$
\mathbf a = (7/4,3/2),\ \mathbf b = (3/2, 3/2),\ \mathbf c = (1,1)
$$
of $\Pi$ using the parsing tree \eqref{23PrsTre} and the vertex/chord Table~\ref{T:VerChTab}.

The first task is to locate the points within the triangles delineated by the decomposition $\delta$. For $\mathbf x=(x_1,x_2)$, we have
\begin{align*}
2\braket{\mathbf x|1\wedge3}=
\begin{vmatrix}
1 &x_1 &x_2\\
1 &2 &1\\
1 &1 &2
\end{vmatrix},
2\braket{\mathbf x|1\wedge5}=
\begin{vmatrix}
1 &x_1 &x_2\\
1 &2 &1\\
1 &0 &0
\end{vmatrix},
2\braket{\mathbf x|3\wedge5}=
\begin{vmatrix}
1 &x_1 &x_2\\
1 &1 &2\\
1 &0 &0
\end{vmatrix},
\end{align*}
yielding the respective codes from the vertex/chord Table~\ref{T:VerChTab} displayed in Table~\ref{T:abcVChTb}. Thus, we have $\mathbf a\in\tau^{\varnothing}$, $\mathbf b\in\tau^{\varnothing}\cap\tau^L$, and $\mathbf c\in\tau^L$.

\renewcommand{\arraystretch}{1.5}
\begin{table}[htb]
\begin{tabular}{||c||c|c|c||c||}
\hline
Chord: &$\ket{1\wedge3}$ &$\ket{1\wedge5}$ &$\ket{3\wedge5}$ &Code\\
\hline
$\mathbf a$ &$\le0$ &$\le0$ &$\ge0$ &001\\
\hline
$\mathbf b$ &$0$ &$\le0$ &$\ge0$ &001,101\\
\hline
$\mathbf c$ &$\ge0$ &$\le0$ &$\ge0$ &101\\
\hline
\end{tabular}
\vskip 4mm
\caption{The codes from the vertex/chord Table~\ref{T:VerChTab} for the vertices $\mathbf a,\mathbf b$ and $\mathbf c$.
}
\label{T:abcVChTb}
\end{table}

Now, we merely follow \eqref{E:Aov0v1v2} and \eqref{E:ArCordFn} to find the areal coordinates of $\mathbf a$ in the simplex $\tau^{\varnothing}$, of $\mathbf b$ in the simplex $\tau^{\varnothing}$, say, and of $\mathbf c$ in the simplex $\tau^L$. This procedure gives the coordinate expressions
\begin{align*}
\mathbf a &= 1/2 \, \mathbf v_1 + 1/4 \, \mathbf v_2 + 1/4 \, \mathbf v_3\,,\\
\mathbf b &= 1/2 \, \mathbf v_1 + 1/2 \, \mathbf v_3\,,\quad \mbox{and}\\
\mathbf c &= 1/3 \, \mathbf v_1 + 1/3 \, \mathbf v_3 + 1/3 \, \mathbf v_5\,,
\end{align*}
the chordal coordinates for all the absent vertices being zero in each case.

\subsection{Chordal coordinates of a polygon}\label{SSS:ChordFun}

Consider a polygon $\Pi$ with ordered vertex set $V=\set{\mathbf v_1<\dots<\mathbf v_n}$. Let $\delta$ be a chordal decomposition of $\Pi$. Suppose that the procedure outlined in \S\ref{SSRegnIdOr} has been executed, thus producing the set $T_\delta$ of $n-2$ correctly oriented triangles $\tau$ appearing in $\delta$. For each triangle $\tau\in T_\delta$, and for each vertex $\mathbf v$ of $\tau$, take the areal coordinate function given by \eqref{E:ArCordFn}.

\begin{definition}\label{D:ChordFun}
(a)
For $S\subset T_\delta$ and $\tau\in T_\delta\smallsetminus S$, define
$$
\bigodot_{\sigma\in\varnothing}\mathds 1_\sigma\cdot\ket{\mathbf v}_\sigma=0
$$
and
\begin{equation}\label{E:RcStOdot}
\bigodot_{\sigma\in S\cup\set\tau}\mathds 1_\sigma\cdot\ket{\mathbf v}_\sigma
=\bigodot_{\sigma\in S}\mathds 1_\sigma\cdot\ket{\mathbf v}_\sigma
+\mathds 1_\tau\cdot\ket{\mathbf v}_\tau
-\mathds 1_{\tau\cap\,\bigcup\set{\sigma|\sigma\in S}}\cdot\ket{\mathbf v}_\tau
\end{equation}
recursively, using the pointwise product in the function algebra $I^\Pi$.
\vskip 2mm
\noindent
(b)
Using the notation of (a), the formula
\begin{equation}\label{E:ChordFun}
\ket{\mathbf v}_\delta
=\bigodot_{\sigma\in\set{\tau\in T_\delta|\mathbf v\in\tau}}\mathds 1_\sigma\cdot\ket{\mathbf v}_\sigma
\end{equation}
gives the \emph{chordal coordinate function} of $\delta$ at a vertex $\mathbf v$ of $\Pi$.
\end{definition}

In connection with the application of the recursion step \eqref{E:RcStOdot} to an element $\mathbf a$ of $\Pi$, we have the following:

\begin{lemma}\label{L:RcStOdot}
Consider the equality
\begin{equation}\label{E:RcTaOdot}
\bra{\mathbf a}\bigodot_{\sigma\in S\cup\set\tau}\mathds 1_\sigma\cdot\ket{\mathbf v}_\sigma
=
\bra{\mathbf a}\bigodot_{\sigma\in S}\mathds 1_\sigma\cdot\ket{\mathbf v}_\sigma
\end{equation}
for $S\subset T_\delta$, $\tau\in T_\delta$, $\mathbf a\in\Pi$, and $\mathbf v\in V$.
\begin{enumerate}
\item[$(a)$]
If $\mathbf a\notin\tau$, then \eqref{E:RcTaOdot} holds.
\item[$(b)$]
If $\braket{\mathbf a|\mathbf v}_\tau=0$, then \eqref{E:RcTaOdot} holds.
\item[$(c)$]
If $\mathbf a$ lies in all $\rho\in S\cup\set\tau$, then \eqref{E:RcTaOdot} holds.
\end{enumerate}
\end{lemma}

\begin{proof}
By \eqref{E:RcStOdot}, we have
\begin{align}\notag
\bra{\mathbf a}\bigodot_{\sigma\in S\cup\set\tau}\mathds 1_\sigma\cdot\ket{\mathbf v}_\sigma
=
\bra{\mathbf a}\bigodot_{\sigma\in S}&\mathds 1_\sigma\cdot\ket{\mathbf v}_\sigma
\\ \label{E:ExRecTrm}
&
+\bra{\mathbf a}\mathds 1_\tau\cdot\braket{\mathbf a|\mathbf v}_\tau
-\bra{\mathbf a}\mathds 1_{\tau\cap\,\bigcup\set{\sigma|\sigma\in S}}\cdot
\braket{\mathbf a|\mathbf v}_\tau\,.
\end{align}

\vskip 2mm
\noindent
(a)
The difference \eqref{E:ExRecTrm} vanishes in this case since $\bra{\mathbf a}\mathds1_\tau$ and $\bra{\mathbf a}\mathds1_{\tau\cap\,\bigcup\set{\sigma|\sigma\in S}}$ are zero.
\vskip 2mm
\noindent
(b)
The difference \eqref{E:ExRecTrm} vanishes in this case since $\braket{\mathbf a|\mathbf v}_\tau=0$.
\vskip 2mm
\noindent
(c)
The difference \eqref{E:ExRecTrm} vanishes in this case since $\bra{\mathbf a}\mathds 1_{\tau\cap\,\bigcup\set{\sigma|\sigma\in S}}=1$ and $\bra{\mathbf a}\mathds 1_\tau=1$.
\end{proof}

\begin{theorem}\label{T:ChrdCord}
Let $\delta$ be a chordal decomposition of $\Pi$.
\begin{enumerate}
\item[$(\mathrm a)$]
The partition of unity property
$$
\sum_{\mathbf{v}\in V}\ket{\mathbf v}_\delta=\mathds 1_\Pi
$$
holds.
\item[$(\mathrm b)$]
With $\mathbf a\in\Pi$, the barycentric or linear precision property
\begin{equation}\label{E:LnPrChrd}
\bra{\mathbf a}=\sum_{\mathbf v\in V}\braket{\mathbf a|\mathbf v}_\delta\bra{\mathbf v}
\end{equation}
holds.
\end{enumerate}
\end{theorem}

\begin{proof}
By Lemma~\ref{L:BPimpPO1}, it suffices to prove (b), since (a) then follows by application of the barycentric homomorphism $\mathds 1\colon\Pi\to I$ to both sides of (b). Consider an element $\mathbf a$ of $\Pi$. There are various cases to consider, according to the location of $\mathbf a$ in the chordal decomposition $\delta$. In the remainder of this proof, ``interior'' refers to relative interiors.

\vskip 2mm
\noindent
\textbf{Interiors of triangles:}
If $\mathbf a$ lies in the relative interior of the oriented triangle $\tau=\mathbf v_i\wedge \mathbf v_j\wedge \mathbf v_k$ of the chordal decomposition, then by \eqref{E:ChordFun} and Lemma~\ref{L:RcStOdot}(a),
\begin{equation}\label{E:UniquTri}
\braket{\mathbf a|\mathbf v}_\delta
=
\bigodot_{\sigma\in\set{\rho\in T_\delta|\mathbf v\in\rho}}\bra{\mathbf a}\mathds 1_\sigma\cdot\braket{\mathbf a|\mathbf v}_\sigma
=\braket{\mathbf a|\mathbf v}_\tau
\,,
\end{equation}
so $\braket{\mathbf a|\mathbf v}_\delta=0$ if $\mathbf v\notin\set{\mathbf v_i,\mathbf v_j,\mathbf v_k}$. In this case, Proposition~\ref{P:ArCordFn} yields \eqref{E:LnPrChrd} as an instance of the usual expression for areal coordinates in a triangle.

\vskip 2mm
\noindent
\textbf{Interiors of boundary edges:}
Suppose that $\mathbf a$ lies in the interior of a boundary edge $\mathbf v_i\wedge \mathbf v_{i+1}$.  The boundary edge forms part of a unique triangle $\tau=\mathbf v_i\wedge \mathbf v_{i+1}\wedge \mathbf v_j$ of the chordal decomposition. Then, as with \eqref{E:UniquTri}, we have $\braket{\mathbf a|\mathbf v}_\delta=0$ if $\mathbf v\notin\set{\mathbf v_i,\mathbf v_{i+1},\mathbf v_j}$, and  Proposition~\ref{P:ArCordFn} yields \eqref{E:LnPrChrd} as an instance of the usual expression for areal coordinates in a triangle.

\vskip 2mm
\noindent
\textbf{Interiors of chords:}
Suppose that $\mathbf a$ lies in the relative interior of the chord $\mathbf v_i\wedge \mathbf v_j$ separating the oriented triangles $\tau_1=\mathbf v_i\wedge \mathbf v_j\wedge \mathbf v_k$ and $\tau_2=\mathbf v_l\wedge \mathbf v_j\wedge \mathbf v_i$. Then by \eqref{E:ChordFun} and Lemma~\ref{L:RcStOdot}(a),
\begin{align*}
\braket{\mathbf a|\mathbf v}_\delta
&
=
\bigodot_{\sigma\in\set{\rho\in T_\delta|\mathbf v\in\rho}}\bra{\mathbf a}\mathds 1_\sigma\cdot\braket{\mathbf a|\mathbf v}_\sigma
\\
&
=
\bra{\mathbf a}\mathds1_{\tau_1}\cdot\braket{\mathbf a|\mathbf v}_{\tau_1}
+
\bra{\mathbf a}\mathds1_{\tau_2}\cdot\braket{\mathbf a|\mathbf v}_{\tau_2}
-
\bra{\mathbf a}\mathds1_{\tau_1\cap\tau_2}\cdot\braket{\mathbf a|\mathbf v}_{\tau_2}
\\
&
=
\braket{\mathbf a|\mathbf v}_{\tau_1}
+
\braket{\mathbf a|\mathbf v}_{\tau_2}
-
\braket{\mathbf a|\mathbf v}_{\tau_2}
=\braket{\mathbf a|\mathbf v}_{\tau_1}
\,,
\end{align*}
so $\braket{\mathbf a|\mathbf v}_\delta=0$ if $\mathbf v\notin\set{\mathbf v_i,\mathbf v_j}$. In this case, \eqref{E:LnPrChrd} reduces to an instance of the usual expression for areal coordinates on a line segment.

\vskip 2mm
\noindent
\textbf{Vertices:}
Suppose that $\mathbf a$ is a vertex $\mathbf v_i$ of $\Pi$. Now, for any triangle $\sigma$ of the chordal decomposition $\delta$, we have $\braket{\mathbf v_i|\mathbf v_j}_\sigma=\delta_{ij}$ with the Kronecker delta. By \eqref{E:ChordFun} and Lemma~\ref{L:RcStOdot}(b), we have $\braket{\mathbf v_i|\mathbf v_j}_\delta=0$ for $i\ne j$.

Now suppose that $\tau_1,\dots,\tau_r$ is the list of triangles within the chordal decomposition $\delta$ that contain $\mathbf v_i$. By Lemma~\ref{L:RcStOdot}(a), we have that
$$
\braket{\mathbf v_i|\mathbf v_i}_\delta=
\bra{\mathbf v_i}\bigodot_{\sigma\in\set{\tau_1,\dots,\tau_r}}\mathds 1_\sigma\cdot\ket{\mathbf v_i}_\sigma\,.
$$
We then have
$$
\bra{\mathbf v_i}\bigodot_{\sigma\in\set{\tau_1,\dots,\tau_k}}\mathds 1_\sigma\cdot\ket{\mathbf v_i}_\sigma=1
$$
for $1\le k\le r$ by induction, using Lemma~\ref{L:RcStOdot}(c). Thus \eqref{E:LnPrChrd} also holds in this final case.
\end{proof}

\begin{example}\label{X:ChordFun}
Let $\Pi$ be the convex hull (in $\mathbb{R}^2$) of the ordered sequence
$$
\mathbf v_1=(0,0)
=\tfrac23\mathbf v_2-\tfrac13\mathbf v_3+\tfrac23\mathbf v_4,
\mathbf v_2=(1,0),
\mathbf v_3=(0,1),
\mathbf v_4=(-1,\tfrac{1}{2})
$$
of extreme points. Let $\mathbf a$ be the point $\sum_{i=1}^4\tfrac14\mathbf v_i=(0,\frac38)$.
Consider the chordal decompositions $\delta=\set{1\wedge 3}$ and $\delta'=\set{2\wedge 4}$ of $\Pi$, with respective region sets
$$
T_{\delta}=\set{\tau_{1}=1\wedge2\wedge3,\tau_{2}=1\wedge3\wedge4}\phantom{\,.}
$$
and
$$
T_{\delta'}=\set{\tau_{1}'=1\wedge2\wedge4,\tau_{2}'=2\wedge3\wedge4}\,.
$$
Since $\mathbf a\in\tau_1\cap\tau_2$, its chordal coordinates with respect to $\delta$ are $\left(\frac58,0,\frac38,0\right)$. On the other hand, since $\mathbf a\in\tau_2'\smallsetminus\tau_1'$, its chordal coordinates with respect to $\delta'$ are $\left(0,\frac5{12},\frac2{12},\frac5{12}\right)$.
\end{example}

\section{Cartographic coordinates of a polygon}\label{SS:CartoCor}

For this section, consider a polygon $\Pi$ with an ordered set $V$ of $n$ vertices, as presented in \S\ref{SSS:CombGeom}. To remove any bias that might have resulted from the choice of any single chordal coordinate system, we average over the action of the dihedral group to yield a \emph{cartographic coordinate system}.

\subsection{Cartographic coordinates for a CDS}\label{SSS:CartoCDS}

\begin{definition}\label{D:CartoCDS}
Let $\delta$ be a chordal decomposition of $\Pi$ with a specified CDS. Then the formula
\begin{equation}\label{E:CartoCDS}
\ket{\mathbf v}_\kappa=
\frac1{2n}\sum_{g\in D_n}\ket{\mathbf v}_{\delta g}
\end{equation}
gives the \emph{cartographic coordinate function} of that CDS at a vertex $\mathbf v$ of $\Pi$.
\end{definition}

In Definition~\ref{D:CartoCDS}, the right action $g\colon\delta\mapsto\delta g$ of a dihedral group element $g$ appearing in \eqref{E:CartoCDS} on the chordal decomposition $\delta$ is as described in \S\ref{SSS:SkGrChrd} and Remark~\ref{R:VerChTab}(c).

\begin{theorem}\label{T:CartoCDS}
For each chordal decomposition $\delta$ of $\Pi$, Definition~\ref{D:CartoCDS} specifies a coordinate system $\kappa$ of $\Pi$.
\end{theorem}

\begin{proof}
In the notation of \S\ref{SS:PolCoSys}, we have
$$
\kappa=\frac1{2n}\sum_{g\in D_n}\delta g\,,
$$
a barycentric combination. By Theorem~\ref{T:ChrdCord}, each $\delta g$ forms a coordinate system. Theorem~\ref{T:PolCoSys} then shows that $\kappa$ is a coordinate system.
\end{proof}

\begin{example}\label{X:CartoCDS}
Continuing Example~\ref{X:ChordFun}, recall that $1^2$ is the only CDS available for a quadrilateral. We obtain
$$
\left(\tfrac5{16},0,\tfrac3{16},0\right)
+\left(0,\tfrac5{24},\tfrac2{24},\tfrac5{24}\right)
=\tfrac1{48}(15,10,13,10)
\simeq(0.312,0.208,0.271,0.208)
$$
as the cartographic coordinates for the point $\mathbf a$.
\end{example}

\subsection{Cartographic coordinates on the hexagon}\label{SSS:MoSamCom}

The hexagon $\Pi$ that was discussed in \S\ref{SSS:SoSamCom} gives a more illuminating illustration of cartographic coordinates that that provided by Example~\ref{X:CartoCDS}. We consider the points $\mathbf a,\mathbf b,\mathbf c$.

To calculate the cartographic coordinates of the points $\mathbf a, \mathbf b$ and $\mathbf c$, first note that that the orbit of the $D_6$-action that contains the initial chordal decomposition $\delta$ consists of just two chordal decompositions: the original $\delta_1 = \delta$, and then its image $\delta_2$ given by $\mathds 1_{24} \circ \mathds 1_{26} \circ \mathds 1_{46}$. Applying \eqref{E:CartoCDS} to $\delta$ from \S\ref{SSS:SoSamCom}, we have
\begin{equation}\label{E:averaged}
|\mathbf v_i\rangle_{\kappa} = \frac1{12} (6 \cdot |\mathbf v_i\rangle_{\delta_1} + 6 \cdot |\mathbf v_i\rangle_{\delta_2}) = \frac1{2} (|\mathbf v_i\rangle_{\delta_1} + |\mathbf v_i\rangle_{\delta_2})
\end{equation}
for $i = 1,\dots ,6$. The chordal coordinates of $\mathbf a, \mathbf b$ and $\mathbf c$ with respect to $\delta_2$ are calculated as in the previous case of $\delta_ 1=\delta$ from \S\ref{SSS:SoSamCom}. Specifically, we obtain
\begin{align*}
&\mathbf a = (0,3/4,0,0,0,1/4),\\
&\mathbf b = (0,2/3,0,1/6,0,1/6),\\
&\mathbf c = (0,1/3,0,1/3,0,1/3)
\end{align*}
for $\delta_2$. Thus, according to \eqref{E:averaged}, the points $\mathbf a, \mathbf b$ and $\mathbf c$ are presented with
\begin{align*}
&\mathbf a = (1/4,1/2,1/8,0,0,1/8),\\
&\mathbf b = (1/4,2/6,1/4,1/12,0,1/12),\\
&\mathbf c = (1/6,1/6,1/6,1/6,1/6,1/6)
\end{align*}
as their cartographic coordinates. In particular, note the more symmetrical presentation of the barycenter $\mathbf c$ of $\Pi$ than that provided by the original chordal coordinates.

\end{document}